\documentclass[12pt,reqno]{amsart}
\usepackage[widespace]{fourier}
\usepackage{amsmath,amscd,amsfonts,amssymb}
\usepackage[all]{xypic}
\usepackage{hyperref}

\numberwithin{equation}{section} 

\DeclareFontFamily{OT1}{rsfs}{}
\DeclareFontShape{OT1}{rsfs}{n}{it}{<-> rsfs10}{}
\DeclareMathAlphabet{\mathscr}{OT1}{rsfs}{n}{it}

\newtheorem{theorem}{Theorem}[section]
\newtheorem{proposition}[theorem]{Proposition}
\newtheorem{lemma}[theorem]{Lemma}
\newtheorem{corollary}[theorem]{Corollary}

\theoremstyle{definition}
\newtheorem{definition}[theorem]{Definition}

\theoremstyle{remark}
\newtheorem{remark}[theorem]{Remark}

\begin{document}

\newcommand{\Hom}{{\mathrm {Hom}}}

\newcommand{\eP}{{\mathscr P}}
\newcommand{\eE}{{\mathscr E}}
\newcommand{\eG}{{\mathscr G}}
\newcommand{\eV}{{\mathscr V}}

\newcommand{\spec}{{\rm Spec}\,}

\newcommand{\Aut}{\mbox{{\rm Aut}$\,$}}
\newcommand{\bc}{{\mathbb C}}
\newcommand{\bq}{{\mathbb Q}}
\newcommand{\br}{{\mathbb R}}

\newcommand{\cG}{\mathcal{G}}
\newcommand{\cE}{\mathcal{E}}
\newcommand{\ce}{{\mathcal E}}
\newcommand{\co}{\mathcal O}

\newcommand{\cb}{\mathcal{B}}

\newcommand{\cg}{{\mathcal G}}
\newcommand{\ca}{{\mathcal A}}
\newcommand{\RR}{\mathbb{R}}
\newcommand{\gfr}{\mathfrak{g}}
\newcommand{\cO}{\mathcal{O}}

\newcommand{\cF}{\mathcal{F}}
\newcommand{\cM}{\mathcal{M}}
\newcommand{\cA}{\mathcal{A}}

\newcommand{\lra}{\longrightarrow}

\newcommand{\PP}{\mathbb{P}}
\newcommand{\ZZ}{\mathbb{Z}}
\newcommand{\GG}{\mathbb{G}}
\newcommand{\QQ}{\mathbb{Q}}
\newcommand{\CC}{\mathbb{C}}
\newcommand{\Id}{\mathrm{Id}}

\title[Connections on parahoric torsors over curves]{Connections on parahoric
torsors over curves}

\author[V. Balaji]{Vikraman Balaji}
\address{Chennai Mathematical Institute, Sipcot IT Park, Siruseri 603103, India}
\email{balaji@cmi.ac.in}

\author[I. Biswas]{Indranil Biswas}

\address{School of Mathematics, Tata Institute of Fundamental
Research, Homi Bhabha Road, Mumbai 400005, India}

\email{indranil@math.tifr.res.in}

\author[Y. Pandey]{Yashonidhi Pandey}

\address{Indian Institute of Science Education and Research, Mohali Knowledge city,
Sector 81, SAS Nagar, Manauli PO 140306, India}

\email{ypandey@iisermohali.ac.in}

\date{}

\keywords{Bruhat--Tits group scheme; parahoric 
torsor; connection; polystability.}

\subjclass[2010]{14F22, 14D23, 14D20.}

\begin{abstract}
We define parahoric $\cG$--torsors for certain Bruhat--Tits group scheme $\cG$ on a 
smooth complex projective curve $X$ when the weights are real, and also define 
connections on them. We prove that a $\cG$--torsor is given by a homomorphism from 
$\pi_1(X\setminus D)$ to a maximal compact subgroup of $G$, where the finite subset $D\, 
\subset\, X$ is the parabolic divisor, if and only if the $\cG$--torsor is polystable.
\end{abstract}

\maketitle

\tableofcontents

\section{Introduction}

Let $X$ be an irreducible smooth projective curve defined over the field of complex 
numbers. Let $G$ be a simple and simply connected affine algebraic group defined over 
$\CC$. Fix a finite subset $D \,\subset \,X$. Let $\cG$ be a Bruhat--Tits group 
scheme with parabolic points at $D$ (see Definition \ref{ht1}).

In \cite{bs}, an analogue of the Mehta--Seshadri theorem in \cite{MS} was proved relating 
stable parahoric torsors under Bruhat--Tits group schemes with irreducible homomorphisms 
of certain Fuchsian groups into a maximal compact subgroup of $G$. This was done under 
the assumption that the parabolic weights are rational, or equivalently, the fixed points 
of the Fuchsian group are all {\em elliptic}. Recall that in \cite{bs} it is shown that 
if the weights are chosen {\em rational} then one can recover the $\cG$--torsors as 
invariant direct images of orbifold principal bundles with respect to suitable ramified 
covers of $X$ ramified over the parabolic points.

An obstruction to cover real weights in the setting of parahoric torsors is that 
the classical Bruhat--Tits group scheme is defined on spectra of discrete valuation 
rings while the phenomenon of parabolic bundles with real weights naturally lies in 
the setting of an analytic neighborhood of the origin. Further, the notion of 
invariant direct image fails to generalize directly to the analytic setting when 
the covering map is no longer algebraic. We address this issue by working with a 
natural analogue of Deligne extensions to the parahoric torsor setting; to a pair $(E,\,
\widehat{d})$ 
of a principal $G$--bundle $E$ equipped with a flat connection $\widehat{d}$ on the 
punctured disc, we give a canonical extension to a torsor under Bruhat--Tits group 
scheme on the compact Riemann surface (see Section \ref{se10.2}). We use this to 
construct the parahoric torsor associated to a representation of the fundamental 
group of the punctured Riemann surface.

Another obstruction in case of real weights is the notion of stability for parahoric 
torsors. Filling the lacuna in \cite{bs}, we have a definition of stability for parahoric 
torsors (\ref{newss}) which covers real weights as well. Then, following the approach in 
the paper of Mehta--Seshadri, \cite[p.~217]{MS}, in Section \ref{variation-ss} we develop 
the theme of {\em variation of weights} in the setting of parahoric torsors, where the 
notions of (semi)stability for the case of real weights are shown to coincide for 
``nearby'' rational weights. Although the broad lines are the same this generalization is 
not entirely straight-forward.

We then go on to define the notion of a connection on a $\cG$--torsor over a smooth 
projective curve over $\bc$ and prove the analogue of the Donaldson--Uhlenbeck--Yau 
correspondence in the parahoric setting when the weights are {\em real}. The basic 
idea is to use the Tannakian formalism and the argument reduces to the case of the 
associated parabolic Lie algebra bundle. The more general reductive group case is 
then an easy generalization.

\subsection{Origins}\label{se.or}

In the early 60s, Mumford defined the notion of stability for vector bundles on curves as 
a tool to get Hausdorff moduli spaces; using geometric invariant theory, he then 
constructed the moduli space of stable vector bundles. In (\cite{NS}) Narasimhan and 
Seshadri gave an alternative characterization of stability using flat connections; more 
precisely, they proved that a holomorphic vector bundle $E$ on a compact Riemann surface 
is stable if and only if $E$ arises from an irreducible projective unitary 
representations of the fundamental group of the Riemann surface. This correspondence 
between flat projective unitary connections and stable vector bundles has been 
generalized in several directions. A. Ramanathan (\cite{Ra}) extended the correspondence 
to the case of holomorphic principal $G$--bundles, where $G$ is a complex reductive 
affine algebraic group \cite{Ra} On the other hand, Mehta and Seshadri (\cite{MS}) 
generalized the Narasimhan--Seshadri construction to include unitary logarithmic 
connections, or equivalently, to classify irreducible unitary representations of general 
Fuchsian groups with fixed conjugacy classes. These logarithmic connections are those 
which have regular singularity at finitely many points and were already apparent in the 
early work of Weil (\cite{We}); their importance was emphasized by Deligne in 
(\cite{del}). The objects replacing stable bundles in the Mehta--Seshadri correspondence 
are stable parabolic vector bundles. Following Donaldson's reinterpretation of the 
Narasimhan--Seshadri correspondence, Biquard (\cite{biquard}) gave a differential 
geometric interpretation of this Mehta--Seshadri correspondence.

It is a very natural problem to generalize the Mehta--Seshadri correspondence from the 
setting of parabolic vector bundles to that of principal $G$--bundles, where $G$ is a 
complex reductive affine algebraic group. On the side of representations, the objects 
were easy to define; they were homomorphisms of Fuchsian groups taking values in a 
maximal compact subgroup of $G$ such that for each puncture of the Riemann surface the 
associated conjugacy class in the fundamental group of the surface is mapped to a fixed 
conjugacy class of the maximal compact subgroup of $G$. Similarly, on the side of 
connections, the corresponding objects were quite well-understood since the work of 
Deligne. The central problem was to generalize the notion of a stable parabolic vector 
bundle to the setting of principal $G$--bundles. From a Tannakian perspective 
(\cite{babina}) it became apparent that a naive generalization in terms {\em parabolic 
$G$--bundles, i.e., principal $G$ bundles with parabolic structures}, was insufficient 
for setting up a comprehensive analogue of the Mehta--Seshadri correspondence.

At a technical level, it was not clear what the correct notion of weight should be in the 
general setting. A breakthrough came in the work of Boalch (\cite{boalch}) and 
Balaji--Seshadri (\cite{bs}). These papers introduced the correct and very natural notion 
of weight, namely, as a point in the affine apartment of $G$. As a consequence of this 
point view, the reason for the inadequacy of parabolic $G$--bundles also became clear. It was 
realized that instead of parabolic $G$--bundles one should consider torsors or principal 
homogeneous spaces under parahoric group schemes in the sense of Bruhat and Tits. Balaji 
and Seshadri (\cite{bs}) extended the Mehta--Seshadri theorem to the case of parahoric 
torsors with rational weights.

Prior to these works, there were at least two partial approaches to generalize the 
Mehta--Seshadri theorem, both around the turn of the millennium. The approach in (\cite{babina}) 
was Tannakian in spirit and followed the method of Nori \cite{nori}; this Tannakian approach identified the basic 
problem, namely that the object associated to a representation of a Fuchsian group into the
maximal compact of $G$, such
that for each puncture of the Riemann surface the associated conjugacy class in the fundamental
group of the surface is mapped to a fixed conjugacy class of the maximal compact subgroup of $G$,
in general will not be a principal $G$--bundle on the Riemann surface. The approach in 
(\cite{tw}) again gave a partial solution to the problem; in the language of Weyl alcoves, the 
solution was for weights in the interior of the Weyl alcove which corresponds to the 
subclass of {\em parabolic $G$--bundles}.

Parahoric group schemes $\cg$ and $\cg$--torsors over smooth projective curves were first 
defined and studied by Pappas--Rapoport \cite{pr}. Subsequently, in \cite{pr2}, they
made several conjectures on the moduli stacks of $\cg$--torsors. Most of 
these conjectures were answered by Heinloth (\cite{He}) providing a precise setting for the study 
of these moduli stacks. The paper of Seshadri (\cite{sesass}), takes up the question of 
the analogue of the Mehta--Seshadri theorem; the main emphasis in (\cite{sesass}) was again was to point 
out that for a solution to the problem of obtaining analogues of the Mehta--Seshadri 
theorem, one has to go beyond the category of principal $G$--bundles. Evidence to the role 
of Bruhat--Tits theory was also given in a few illustrative examples. The paper of Boalch
(\cite{boalch}) studies logarithmic connections on $G$--bundles and the notion of 
parahoric torsors comes along with the first appearance of the notion of weights for 
these torsors. Around the same time and independently in \cite{bs}, a similar notion of 
weights was defined towards providing a satisfactory analogue of the Mehta--Seshadri 
theorem in the general setting for semisimple groups $G$, thereby completing the broad 
picture outlined in \cite{sesass}. The notion of ``invariant direct images" of torsors in 
\cite{bs} plays a key role analogous to the one in \cite{MS} for the case of vector 
bundles.

\section{Preliminaries}\label{prelim}

In this section we collect together some standard notions and notation 
that will be used throughout this paper. See \cite{bt1}, \cite{bt}, \cite{bs},
\cite{He} for this section and the next one.

The base field will always be $\mathbb C$.

Define
\begin{equation}\label{z1}
A\, :=\, \CC[[t]]\ \ \ \text{and }\ \ \ K\,:=\, \CC((t))\, ,
\end{equation}
where $t$ denote a uniformizing parameter. Let $G$ be a {\em semisimple simply
connected affine algebraic group} defined over $\CC$. The Lie algebra of $G$
will be denoted by $\mathfrak g$. Fix a maximal torus $T \,\subset\, G$; let
$Y(T)\,=\, \Hom(\GG_m,\, T)$ be the group of all
holomorphic one--parameter subgroups of $T$.

\subsection{Apartment of $G$}\label{se2.1}

For each maximal torus $T$ of $G$, we have the {\it
standard affine apartment} $\ca_{_T}$. It is an affine space under $Y(T)
\otimes_\ZZ \RR$. In general, there is no origin in the apartment (cf. \cite{bt1}). But
for purposes of this paper, {\it we shall identify $\ca_{_T}$ with $Y(T) \otimes_\ZZ \RR$}
(see \cite[\S~2]{bs}). 

Let ${\mathcal V}$ be a real affine 
space. A function $f\,:\, {\mathcal V}\,\longrightarrow\,
 \RR$ is said to be an {\it affine functional} if 
$$f(rx + (1-r)y)\,=\,rf(x)+(1-r)f(y)$$ 
for all $x\, ,y \,\in\, \mathcal{V}$ and $r \,\in\, \RR$.
Thus, for a root $\alpha$ of $G$ and an integer $n \,\in\, \ZZ$ there is the affine
functional $$\alpha + n\,:\, \ca_{_T}\,\longrightarrow\, \RR\, , \ \ x \,\longmapsto\, 
\alpha(x-0)+n\, .$$ We note that these are called the {\it affine roots of $G$}. The
zero locus of $\alpha + n$ will be denoted by $H_{\alpha+n}$; it is called an 
{\it affine wall}. The set of affine walls is known to be {\it locally finite}, meaning any compact 
subset of $\ca_{_T}$ intersects only finitely many affine walls. For any point $x \,
\in\, \ca_{_T}$, let $Z_x$ denote the set of affine functionals vanishing on $x$. For an integer
$n \,\geq \, 0$, define
$$\mathcal{H}_n\,=\, \{ x \in \ca \,\mid\, \quad |Z_x|\,=\,n\}\, ,
$$
which is the set of all points where $n$ of the affine functionals vanish. A {\it facet} $\Omega$ of $\ca_{_T}$ is defined to be a connected component of $\mathcal{H}_n$ for some $n$. The {\it dimension} of a facet is its dimension as a real manifold. We then have a decomposition
of the apartment
\begin{equation}\label{simpdec}
\ca_{_T} \,=\, \bigsqcup_n \mathcal{H}_n\, .
\end{equation}

Although, as mentioned above, almost always $\Theta$ will be a point of $\ca_{_T}$, 
sometimes $\Theta$ will also be a facet of $\ca_{_T}$.

\section{Parahoric group scheme and torsors}

\subsection{Invariant direct image}\label{invdirimage}

The base field is $\mathbb C$.
Let $p\,:\,W \,\longrightarrow\, U$ be a finite flat surjective morphism of 
normal, integral Noetherian schemes which is Galois. So the Galois group 
$\text{Gal}(p)$, which we will denote by $\Gamma$, acts on $W$ with $U \,=\, 
W/{\Gamma}$ being the quotient. Such a morphism $p$ is called a {\em Galois covering} with Galois 
group $\Gamma$.

Let $\eG$ be an {\em affine group scheme} over $W$. For the above Galois covering map
$p$, the direct image $p_{_*}\eG$ is defined as follows: for each $U$--scheme $S$, set
\begin{equation}
(p_{_*}\eG)(S) \,=\, Hom_{_{W}}(S \times_U W,\,\eG)\, ;
\end{equation}
this is representable by a group scheme \cite[Theorem 4 and Proposition
6]{blr}. Assume that the group scheme $\eG$ is equipped with an action of
the Galois group $\Gamma$
that lifts the action of $\Gamma$ on $W$; in particular, the ``multiplication map'' and
the ``inverse map'' on $\eG$ are $\Gamma$--equivariant. Such a $\eG$ will be called
a {\em $\Gamma$--group scheme} over $W$.

There is a left
action of $\Gamma$ on $S \times_U W$ induced by the action of $\Gamma$ on $W$. 
This and the left action of $\Gamma$ on $\eG$ together induce the following right action of
$\Gamma$ on $(p_{_*}\eG)(S)$:
\begin{equation}
(f.\gamma)([s,w])\,:=\, \gamma^{-1}.f({\gamma}.[s,\,w])\, ,\ \ [s,w] \,\in\, S \times_U W\, 
, \ \gamma \,\in\, \Gamma\, .
\end{equation}
Consider the {\em fixed point} subscheme under the above action of $\Gamma$
on $p_{_*}\eG$. The
general results on fixed point subschemes given in \cite[Section 3]{bass} can be applied
to our situation {\em since the characteristic is $0$}. Consequently, a
canonically defined {\em smooth} closed $X$--subgroup scheme
$$p^{\Gamma}_{_*}\eG \,:= \,(p_{_*}\eG)^{^\Gamma} \,\subset\, p_{_*}\eG$$
is obtained. This $p^{\Gamma}_{_*}\eG$ is representable because $p_{_*}\eG$ is representable.

\begin{definition}[{Invariant direct image}]\label{invariantdirectimage}
 Let $p\,:\,W \,\longrightarrow\, U$ be as above, and let $\Gamma \,= {\rm Gal}(W/U)$.
Let $\eG$ be a smooth affine $\Gamma$--group scheme over $W$. Define the
{\em invariant direct image} of $\eG$ to be
\begin{equation}
p^{\Gamma}_{_*}(\eG)\,:=\, (p_{_*}\eG)^{\Gamma}\, ,
\end{equation}
so $(p^{\Gamma}_{_*}\eG)(S)\,=\, \eG(S \times_U W)^{\Gamma}$ for any $U$--scheme $S$. 

More generally, if $E$ is any affine scheme over $W$ with a lift of the
$\Gamma$--action on $W$, then define the {\em invariant direct image} of $E$ to be
$$p^{\Gamma}_{_*}E\,:=\, (p_{_*}E)^{\Gamma}\, .$$
\end{definition}

\subsection{Parahoric torsors}

Notation of Section \ref{se2.1} will be followed.
Let $R\,=\,R(T,\,G)$ denote the root system of $G$ (cf. \cite[p. 125]{springer}). Thus for
every $r \,\in\, R$, there is
the root homomorphism $u_r \,: \, \GG_a\,\lra\, G$ \cite[Proposition 8.1.1]{springer}. 

For any non-empty subset
$\Theta\,\subset\, \ca_T$, the {\it parahoric subgroup} ${\eP}_{_\Theta}\,\subset\, G(K)$
\begin{equation}
{\eP}_{_\Theta} \,:=\, \langle T(A),\,\ \{u_{r}(t^{m_r(\Theta)}A)\}_{r \in R} \rangle\, .
\end{equation}
is the subgroup generated by $T(A)$ and $\{u_{r}(t^{m_r(\Theta)}A)\}_{r \in R}$,
where $$m_r\,=\,m_r(\Theta) \,=\, - [{\rm inf}_{\theta \in \Theta} (\theta,r)]\, ,$$
and $A$ is defined in \eqref{z1} \cite[Page 8]{bs}. Moreover,
by \cite[Section 1.7]{bt} we have an affine flat smooth group scheme $\cG_{_\Theta}\,\lra\,
\spec(A)$ corresponding to $\Theta$. The set of $K$--valued (respectively,
$A$--valued) points of $\cG_{_\Theta}$ is identified with $G(K)$ (respectively,
$\mathscr{P}_{_\Theta}$). The group scheme $\cG_{_\Theta}$ is uniquely determined by its
$A$--valued points. Here we will often take $\Theta$ to be just a point of $\ca_{_T}$.
\begin{remark} We remark that the notion of a parahoric subgroup is defined in the greatest generality in the basic papers of Bruhat and Tits. For our purposes, the definition given above is sufficient.\end{remark}

Let $X$ be a smooth complex projective curve. We fix once for all a nonempty finite set
of closed points 
\begin{equation}\label{parabolicpoints}
D \,= \,\{ x_{_j} \}_{j=1}^m\, \subset\, X\, .
\end{equation}
These will play the role of parabolic points. Let $A_{_j}\,:=\,
\widehat{\co_{_{X,x_{_j}}}}$ be the complete discrete valuation ring with
function field $K_{_j}\,\simeq\, \bc((t))$ and {\em residue field}\ $\bc$, obtained
by completing the local rings $\cO_{X,x_j}$. We shall denote $\spec(A_j)$ by $D_j$.

\begin{definition}\label{ht1} Let $\cG$ be a flat, affine group scheme on $X$ of finite
type. We call $\cG$ a {\em Bruhat--Tits group scheme} with parabolic points $D$, if
\begin{enumerate}
\item restricted to $X \setminus D $, it
is isomorphic to the split group scheme $G \times (X \setminus D)$, and

\item $\cG|{_{_{D_j}}}\,\longrightarrow\, D_j$
is a Bruhat--Tits group scheme for each $j$.
\end{enumerate}
We shall denote $\cG$ by $\cG_\Omega$, where
$\Omega \,=\, \{\Omega_j\}_{j=1}^{m}$ is a collection of facets of the
Bruhat--Tits building with $\cG|_{D_j}$ corresponding to $\Omega_j$.
\end{definition}
By the general theory due to Bruhat and Tits, one has an affine flat smooth group scheme 
$\cg_{_\Theta}$ on $\spec(A)$ corresponding to $\Theta$. The set of $K$--valued 
(respectively, $A$--valued) points of $\cg_{_\Theta}$ is identified with $G(K)$ 
(respectively, ${\eP}_{_\Theta}$). The group scheme $\cg_{_\Theta}$ is uniquely 
determined by its $A$--valued points. These notions were defined and the moduli stack of 
$\cg_{_\Theta}$--torsors studied extensively in the papers of Pappas--Rapoport (cf. 
\cite{pr}, \cite{pr2}) and Heinloth (\cite{He}). To construct one on the whole of $X$ one 
can proceed as in \cite[5.2]{bs}. Existence of such group schemes also follows from the 
invariant direct images constructed above (see also \cite[Theorem 5.2.7]{bs}).

\begin{definition}[{\cite[Section 6]{bs}}]\label{quasiparahoric}
A {\it quasi--parahoric} torsor $\cE$ is a ${\mathcal 
G}_{_{\Omega,X}}$--torsor on $X$. \end{definition}

\begin{definition}[{\cite[Section 6]{bs}, \cite{boalch}}]\label{parahorictorsor}
A {\it parahoric torsor} is a pair $(\cE ,\, 
{\boldsymbol\theta})$ consisting of
\begin{enumerate}
\item a ${\mathcal G}_{_{\Omega,X}}$--torsor $\cE\,\longrightarrow\, X$, and
\item weights, meaning elements ${\boldsymbol\theta}\,=\, \{\theta_i\}_{i=1}^m
\,\in\, (Y(T)\otimes \br)^m$ in the interior of $\Omega_i$.
\end{enumerate}
\end{definition}

\begin{remark}\label{classicalwts} 
The above notion of weight is the precise analogue of the classical weight for a 
parabolic vector bundle with {\em multiplicity} (cf. \cite[page 211, Definition 
1.5]{MS}).
\end{remark}

\begin{remark}\label{usedlater}
It should be noted that as in \cite{bs}, the theory of Bruhat--Tits group schemes 
that are used here assumes that the group $G$ is semisimple and simply connected. 
On the other hand, for the case of ${\rm GL}(V)$, which satisfies neither of these 
conditions, the parabolic bundles is classical (cf. \cite{pibundles}, \cite{MS},
\cite{bis}). In \cite[Example 2.3.4]{bs} and \cite[Remark 6.1.5]{bs}, it is noted that the 
torsors under Bruhat--Tits group schemes for ${\rm GL}(V)$ are same 
as the parabolic vector bundles. Let us spell this out for 
the convenience of the reader. Let $\mathcal{GL}(V)$ be a Bruhat--Tits group 
scheme on $\spec(A)$ with generic fiber ${\rm GL}(V)$. Fix a maximal torus 
$T(V)\,\subset\,{\rm GL}(V)$. Let $\ce$ be a $\mathcal {GL}(V)$--torsor, and let 
$\theta(V)\,\in\, Al(T(V))_{_\br}$ be a weight as a point in the so-called Weyl 
alcove (see \cite[p. 9]{bs}). Then, the associated vector bundle gets a canonical 
parabolic structure with quasi--parabolic type determined by the group scheme 
$\mathcal{GL}(V)$ while the parabolic weights are given by $\theta(V)$. We observe 
that in the case of ${\rm GL}(n, {\mathbb C})$, or ${\rm SL}(n, {\mathbb C})$, 
giving a point $\theta(V)\,\in\, Al(T(V))_{_\br}$ in the Weyl alcove is 
equivalent to giving $n$--tuples $(\alpha_{_1},\, \alpha_{_2},\, \cdots ,\,
\alpha_{_n})\,\in \,\br^{^{n}}$, such that every $\alpha_{_i} \,\geq\, 0$ and 
$\alpha_{_i}\,\leq\, \alpha_{_{i+1}}$, for $1 \,\leq\, i \,\leq\, n-1$; in the 
case of ${\rm SL}(n, {\mathbb C})$, the condition of the determinant translates to 
the further condition $$\sum_{i = 1}^{n} \alpha_{_i} \,\in\, {\mathbb Z}\, .$$ By 
a Weyl group conjugation, we can also arrange the $\alpha_{_i}$'s in increasing 
order, i.e., $0 \,\leq\, \alpha_{_1} \,\leq\, \alpha_{_2} \,\leq\,\cdots \,\leq 
\,\alpha_{_n} \,<\, 1$. Once the parahoric subgroup is chosen, this determines a 
flag type and hence we may even order the parabolic weights in a strictly 
increasing sequence.
\end{remark}

\subsection{Rational weights}\label{rationalweights}

Starting with a $m$--tuple of weights ${\boldsymbol\theta}\,\in\, (Y(T) \otimes \bq)^m$, 
following the proof of the converse in \cite[Theorem 2.3.1]{bs}, we get positive 
integers $d_1,\, d_2,\, \cdots ,\, d_m$ such that $d_i\cdot\theta_i \,\in\, Y(T)$. 
By choosing these $d_i$ to be {\em smallest with this property}, we see that a choice 
of $\boldsymbol\theta$ entails a choice of ramification index $d_x$ at each point 
$x\, \in\, D$.

There exists a ramified Galois cover of {\it curves} $p\,:\,Y \,\lra\, X$ which is
\begin{itemize}
\item unramified
over $X\setminus D$, and

\item the ramification index over $x\,\in\,D$ is $d_x$,
\end{itemize}
if and only if exactly one of the following conditions hold:
\begin{enumerate}
\item the genus of $X$ is nonzero,

\item $X\,=\,\PP^1$ and $\# D \,\geq\, 3$,

\item $X\,=\,\PP^1$, $\# D \,=\, 2$, and $d_x\,=\, d_y$, where $D\,=\, \{x,\, y\}$.
\end{enumerate}
(See \cite[p. 26, Proposition 1.2.12]{Na}.)

When we consider parahoric torsors as invariant direct images of 
$(\Gamma,G)$ bundles, it will always be with respect to such a Galois cover.
See also \cite{bis}.

\subsection{Rational weights and parahoric $\cG$--torsors as
$(\Gamma,G)$--bundles}\label{mtbs}

Let $(X,\,D)$ be as above, and let $\cG \,\lra\, X$ be a
Bruhat--Tits group scheme over $X$. By
\cite[Theorem 5.3.1]{bs}, there exists a (possibly ramified) finite Galois cover $p
\,:\, Y \,\lra\, X$
branched over $D$, and a principal $G$--bundle $F$ over $Y$ (cf. \cite[Notation 5.1.0.1]{bs}) equipped
with a lift of the action of the Galois group $\Gamma\, :=\, \text{Gal}(p)$, such
that the following statements hold:
\begin{enumerate}
\item Let $_FG\, :=\, Isom_Y(F,F)$ be the twisting of the constant group scheme
$G\,\lra\, Y$ by $F$. The invariant direct image satisfies the condition 
\begin{equation}\label{gpscheqn} p_*^{\Gamma} {}_FG \,=\, \cG\, .
\end{equation} 

\item Let $D_Y\, :=\, p^{-1}(D)\, \subset\, Y$ denote the ramification points of
the covering $p$. For each $y \,\in\, D_Y$, let $\Gamma_y\, \subset\, \Gamma$ be the isotropy
subgroup that fixes $y$. Let $\tau_y \,:\, \Gamma_y \,\lra\, \Aut(F_y)$ be the action of
$\Gamma_y$ on the fiber $F_y$. Its conjugacy class is called the local type of $F$ at
$y$; this local type will be denoted by $[\tau_y]$. By the type $\tau$ of $F$, we shall mean the set of conjugacy classes $[\tau_y]$ of $\tau_y$:
$$\tau \,:= \,\{[\tau_y]\,\mid\, y \,\in\, D_Y\}.$$ 
Let $\cM_Y^{\tau}(\Gamma, G)$ denote the moduli stack of 
$(\Gamma,G)$ bundles over $Y$ of type $\tau$, and let $\cM_X(\cG)$ denote the
moduli stack of $\cG$ torsors on $X$.
Then there is an isomorphism of moduli stacks 
\begin{equation}\label{sbs}
\alpha_F\,:\, \cM_Y^{\tau}(\Gamma, G) \,\lra \,\cM_X(\cG)\, ,
\end{equation}
given by the $(\Gamma,G)$ bundle $F$ as follows:
Denote by $F^{op}$ the left
$G$--bundle defined by $gf\,: = \,f g^{-1}$, where $g\, \in\, G$ and $f\, \in\, F$. The
above group scheme ${}_FG$ acts on the right of $F^{op}$. The isomorphism in \eqref{sbs} is:
\begin{equation}\label{updown}
E \,\longmapsto \,p_*^\Gamma(E \times_{Y,G} F^{op}) \,=\, p_*^\Gamma(Isom_Y(E,F)).
\end{equation}
The inverse of the map in \eqref{sbs} is given by 
\begin{equation}\label{downup}
\cE \,\longmapsto\, p^*\cE \times_{p^*\cG} F\, .
\end{equation}

\item Let $y \,\in\, D_Y$, and $x\,:=\,p(y)$. Let $N_y\,=\,{\rm Spec}(B)$, where $B\,=\,
\widehat{\cO_{Y,y}}$, and also $D_x\,=\,{\rm Spec}(A)$ with $A\,=\,\widehat{\cO_{X,x}}$. Let
$U_y$ denote the group of local $\Gamma_y$--$G$ automorphisms of $F|_{N_y}$ (cf.
\cite[Definition 2.2.7]{bs}). Then by \cite[Proposition 5.1.2]{bs} we have
\begin{equation}\label{lug}
\cG|_{D_x}(A)\,=\,U_y\, .
\end{equation}
For the equality in \eqref{lug}, we need the existence of $(\Gamma,G)$ bundle $F$ only
locally on $N_y$ and not on entire $Y$.
\end{enumerate}

A different $(\Gamma,G)$ bundle $F'$ of type $\tau$ will, in general, provide a 
different (cf. \eqref{sbs}) isomorphism $\alpha_{F'}\,:\, \cM_Y^{\tau}(\Gamma,G)\,
\lra \, \cM_X(\cG)$.

\section{Change of weights under a homomorphism}\label{variation-rep}

\subsection{The local homomorphism problem}\label{lv}

Consider the following problem. Let $A$ be an arbitrary discrete valuation ring with 
quotient field $K$. Let $G$ be semisimple and simply connected group. Let $\rho\,:\, G
\,\lra\, H$ 
be a homomorphism. Fix a maximal torus $T 
\,\subset\, G$. Then fix a weight $\theta$ in the affine apartment $\ca_{_T}$ for $T$. It
should be emphasized that $\theta$ may not be a rational point, meaning it may not
lie in the image of $Y(T) \otimes_\ZZ \QQ$. Fix a maximal torus $T_{H}\,\subset
\, H$ such that $\rho(T)\, \subset\, T_H$. Via the identification between $\ca_{_T}$ and
$Y(T) \otimes_\ZZ \RR$ mentioned in Section \ref{se2.1}, we have a linear map
\begin{equation}\label{map-apt}
\ca_{_T} \,\lra\, \ca_{_{T_H}}
\end{equation}
between the apartments. Let $\theta_H$ denote the image of $\theta$ under this map. We 
wish to construct, through $\rho$, a canonical homomorphism of group schemes over 
$\spec(A)$: from $\cG_\theta$ corresponding to $\theta$ to $\mathcal{G}_{\theta_H}$ 
corresponding to $\theta_H$.

\subsection{Facets of a homomorphism}\label{facets}

Let $\rho\,:\, G\,\lra\, H$ be a homomorphism. The affine roots of $H$ give affine 
functionals on $\ca_{_{T_H}}$ (cf. Section \ref{prelim}). These functionals are defined 
over the rationals. By (\ref{map-apt}) it follows that the linear map of apartments 
$\ca_{_T}\,\lra\, \ca_{_{T_H}}$ corresponding to $\rho$ is also defined over the 
rationals. Indeed, it is induced by the algebraic map $\rho\vert_T\, :\, T\,\lra\, T_H$. 
Via $\ca_{_T}\,\lra\, \ca_{_{T_H}}$, we view the above affine functionals on 
$\ca_{_{T_H}}$, associated to the affine roots of $H$, as affine functionals on 
$\ca_{_T}$. Take the union of these affine functionals with the usual affine functionals 
on $\ca_{_T}$ corresponding to the affine roots of $G$. Note that all these functionals 
are defined over rationals. We shall call them as $\rho$--functionals.

\begin{definition} An {\it affine wall}
of $\rho$ is the zero locus of a $\rho$--functional. For any $x\,\in\,
\ca_{_T}$, define $Z_x$ to be the set of all affine functionals
on $\ca_{_T}$ vanishing at $x$. Define
$$\mathcal{H}_n\,:=\, \{ x\,\in\,\ca\,\mid\, \quad |Z_x|\,=\,n\}\, .$$
A {\it facet} of $\rho$ is a connected component of $\mathcal{H}_n$ for some $n\,\geq\, 0$.
\end{definition}

Notice that for the identity homomorphism $G\,\lra\, G$, we have just got the usual 
facets. By the theory of buildings, one knows that the usual affine walls of $G$ 
provide a decomposition of $\ca_{_T}$ (cf. (\ref{simpdec})).

We claim that more generally, the facets of a homomorphism $\rho\,:\, G\,\lra\, H$ 
provide an even finer decomposition of $$\ca_{_T}\,=\, \bigsqcup_n \mathcal{H}_n\, ;$$ 
here $n$ varies over a finite set. To prove this claim, first note that the set of all 
$\rho$--affine walls corresponding to $\rho$--functionals form a locally finite set. 
Indeed, this follows because any compact set $C$ of $\ca_{_T}$ meets finitely many usual 
affine walls in $\ca_{_T}$. Now under the map $\ca_{_T}\,\lra\, \ca_{_{T_H}}$, the image 
of $C$ being compact, meets finitely many affine walls of $\ca_{_{T_H}}$. Thus $C$ meets 
only finitely many walls of $\rho$.

\begin{proposition}\label{close}
Let $\rho\,: \,G \,\lra\, H$ be a homomorphism of algebraic groups. Given a weight
$\theta\,\in\, \ca_{_T}$, there exists a rational weight $\eta \,\in \,\ca_{_T}$ lying in the
same facet as $\theta$ such that $\eta_H$ and $\theta_H$ also lie in a
common facet of $\ca_{_{T_H}}$.
\end{proposition}

\begin{proof} 
It can be shown that if the element $\theta \,\in\, \ca_{_T}$ lies in a zero dimensional 
facet, then it must be a rational point. Indeed, this follows from that fact that we are 
looking at common zero locus of a set of rational functionals. Taking contrapositive of 
the last statement, if $\theta$ is not rational then a $\rho$--facet of $\theta$ cannot 
be zero dimensional. So we can find a rational weight $\eta$ in it.

Now by construction, both $\eta$ and $\theta$ lie in the same facet of $\ca_{_T}$, and also 
$\theta_H$ and $\eta_H$ lie in a common facet of $\ca_{_{T_H}}$.
\end{proof}

\begin{remark}\label{closerational} Henceforth, when we say that a {\it rational weight
$\eta$ is close to $\theta$ with respect to $\rho\,~:~\, G\,\lra~\, H$},
we mean a {\it rational} weight in the sense of Proposition \ref{close}. 
\end{remark}

The constructions in the proof
of Proposition \ref{close} can be extended to the context of a finite set
of representations $\{ \rho_i\,:\, G\,\lra\, {\rm GL}(V_i) \}_{i \leq n}$ of $G$. 

Now further assume that $\eta$ and $\theta$ are actually interior points of
one $\rho$--facet. This assumption implies that
they define isomorphic Bruhat--Tits group schemes with generic fiber $G$ and also 
group schemes with generic fiber $H$. There are canonical homomorphisms between 
them by the following proposition.

\begin{proposition}\label{funct}
Let $\rho\,:\, G\,\lra\, H$ be given. Let $\theta \,\in\, \ca_{_T}$ be a weight with image
$\theta_H$ under the map of apartments $\ca_{_T}\,\lra\, \ca_{_{T_H}}$. Then there is a canonical
homomorphism of group schemes $\cG_{_\theta} \,\lra\, \mathcal{G}_{_{\theta_H}}$ over $\spec(A)$.
\end{proposition}

\begin{proof}
Let us recall a characterization of parahoric groups
when the weight $\eta$ lies in $\ca_{_\bq}$ \cite[Theorem 2.3.1]{bs}. In this case, by
\eqref{lug}, there 
exists a ramified Galois cover $$p\,:\,\spec(B)\,\lra\, \spec(A)$$ with Galois group
say $\Gamma$, and a $(\Gamma,G)$ bundle $F \,\lra\, \spec(B)$, such that
\begin{equation}\label{Avp}
\cG_{_{\Omega(A)}}\,=\, \cG_{_\eta}(A) \,=\, {\rm Aut}_{_{(\Gamma,G)}}(F) \, .
\end{equation}

Choose a rational weight $\eta$ close to $\theta$ with respect to $\rho$. This gives an 
element $\eta_H$ close to $\theta_H$. Further, let $$F(H)\,:=\,F \times_G H$$ be the 
$(\Gamma,H)$ bundle obtained by extending structure group using $\rho$. Via the natural 
homomorphism ${\rm Aut}_{_{(\Gamma,G)}}(F) \, \lra \,{\rm Aut}_{_{(\Gamma,H)}}(F(H))$ and 
the equalities $ \cG_{_\eta}(A)\,=\, {\rm Aut}_{_{(\Gamma,G)}}(F)$ and $ 
\cG_{_{\eta_H}}(A)\,=\,{\rm Aut}_{_{(\Gamma,H)}}(F(H))$, we obtain a map of parahoric 
groups $$ \rho_K\,:\, \cG_{_\eta}(A) \,\lra\, \cG_{_{\eta_H}}(A)\, .$$ Now the 
characterizing property of the Bruhat--Tits group schemes is that they are {\em 
\'etoff\'e}, which means that any morphism at the level of group schemes over $\spec(A)$ 
is determined completely by the $A$--valued points alone (cf. \cite[Definition 
1.7.1]{bt}).

Thus we get a morphism over $\spec(A)$ of group schemes 
\begin{equation}
\mathcal{G}_{_\eta}\,\lra\,\mathcal{G}_{_{\eta_H}}
\end{equation}
extending (\ref{Avp}). Since $\eta_H$ is close to $\theta_H$,
it follows that $\cG_{_{\eta_H}}(A)\,=\,\cG_{_{\theta_H}}(A)$, and hence
we have an induced isomorphism of group 
schemes $\cG_{_{\eta_H}}\,\lra\, \cG_{_{\theta_H}}$ on $\spec(A)$ which gives 
the homomorphism $\rho_{_\theta}\,:\,\mathcal{G}_{_\theta}\,\lra\, \cG_{_{\theta_H}}$.
\end{proof}

To the best of our knowledge, Proposition \ref{funct} is not available in the papers of 
Bruhat and Tits. Alternatively it can be proved using the general framework of 
functoriality of buildings as in \cite[Theorem 2.1.8]{L} and \cite[Theorem 2.2.1]{L}.

\section{Associated Constructions}\label{ac}

As before, $G$ is semisimple and simply connected. Let
$\rho \,:\, G\,\lra\, {\rm GL}(V)$ be a rational representation. Fix a maximal torus
$T \,\subset\, G$, and also fix a maximal torus 
$T_V \,\subset\, {\rm GL}(V)$ containing $\rho(T)$. Let $\ce$ be a parahoric
$\cg_{_{\boldsymbol\theta,X}}$--torsor on $X$ with weights $\boldsymbol\theta$. 

\subsection{The construction}\label{cons-ext-stgp}

As in (\ref{parabolicpoints}), fix a finite subset $D\,=\,\{x_j\}_{j=1}^m \, 
\subset\, X$. For each parahoric group scheme $\cg_{_j}\,\lra\, D_{_j} \,=\, 
\spec(A_{_j})$, $1\, \leq\, j\, \leq\, m$, we get a facet $\Omega_{_j}$ in the 
Bruhat--Tits building $ \cb(G)$. Any weight $\boldsymbol\theta\,\in\, \ca_{_{T}}^m$ 
induces a weight ${\boldsymbol\theta_V}$ in $\ca_{_{T(V)}}^m$ (see Section 
\ref{facets}).

The group scheme ${{\cg}_{_{{\boldsymbol\theta},X}}}$ on $X$ is obtained by
gluing $\cg_{_{\theta_{_j}}}$, for every $1\, \leq\, j\, \leq\, m$, with
$$\cg|_{_{X \setminus D}}\,\simeq
\,(X \setminus D) \times G\, .$$ Using the same gluing data via the representation $\rho$ one immediately gets a 
Bruhat--Tits group scheme $\mathcal{G}\mathcal{L}_{\theta_V,X}$, and using locally Proposition \ref{funct}, we see that $\rho$ gives a natural global homomorphism of group schemes over $X$ 
\begin{equation}\label{assocparabstructure}
\rho_{_X}\,:\,{{\cg}_{_{{\boldsymbol\theta},X}}} \,\lra\,
\mathcal{G}\mathcal{L}_{\theta_V,X}\, .
\end{equation}
Using $\rho_{_X}$ one gets the standard construction of extension of structure groups. 

By Remark \ref{usedlater}, via $\rho$ one obtains an associated parabolic vector
bundle $$\cE_{V,\boldsymbol\theta_V}\,:=\, (\ce(V),\, \boldsymbol\theta_V)$$
with parabolic weights $\boldsymbol\theta_V$. We will mostly need to apply associated
constructions under the adjoint homomorphism
$Ad\,:\, G \,\lra\, {\rm GL}(\gfr)$. 

\subsection{Tensor product of parabolic vector bundles}

Assume that we have homomorphisms $\rho_1\,:\, G \,\lra\, {\rm GL}(V)$ and $\rho_2\,:\,
G \,\lra\, {\rm GL}(W)$, and let $$\rho_1 \otimes \rho_2\,:\, G\,\lra\, {\rm GL}
(V \otimes W)$$ be their tensor product. Therefore, for a parahoric torsor $(\cE,\,
\theta)$ we get parabolic vector bundles $\cE_{V,\boldsymbol\theta_V}$,
$\cE_{W,\boldsymbol\theta_W}$, and $\cE_{V \otimes W,\boldsymbol\theta_{V \otimes W}}$.

When the weights $\boldsymbol\theta$ are rational numbers then, by $(\Gamma,G)$ bundle 
theory, we have a canonical isomorphism of parabolic vector bundles \cite{MY}: 
$$\cE_{V,\boldsymbol\theta_V} \otimes^p \cE_{W,\boldsymbol\theta_W} \,\simeq\, 
\cE_{V \otimes W,\boldsymbol\theta_{V \otimes W}}\, ,$$ where $\otimes^p$ is 
parabolic tensor product. The following proposition, which extends this isomorphism to 
the case of {\it real} weights, is immediate by observing that the quasi--parabolic 
bundle for a parabolic tensor product does not change under sufficiently small
change of the parabolic weights of the factors, while the parabolic weights
of the parabolic tensor product are given by the parabolic weights of the factors
using a standard algebraic formula.

\begin{proposition}\label{canisoparten}
 Let $\theta \,\in\, \ca_{_T}$ be arbitrary. Then there is a canonical isomorphism
$$
\cE_{V,\boldsymbol\theta_V} \otimes^p \cE_{W,\theta_W}\,\simeq\,
\cE_{V \otimes W,\theta_{V \otimes W}}
$$
of parabolic vector bundles.
\end{proposition}

Let $V$ and $W$ be two representations of $G$ and $\psi\,:\, V\,\lra\, W$ a
$G$--equivariant homomorphism. Let $(\ce,\, \boldsymbol\theta)$ be a parahoric
torsor. By identifying the $G$--module $Hom(V,W)$ with $W \otimes V^*$ it follows
that $\psi$ induces a homomorphism
\begin{equation}\label{ass-con-real}
(\cE,{\boldsymbol\theta})(\psi)\,:\,\cE_{V,\boldsymbol\theta_V}\,\lra\, \cE_{W,\boldsymbol\theta_W}\, .
\end{equation}

\subsection{Lie algebra bundle of a $\cG$-torsor}\label{lub}

Consider a parahoric $\cG_{\boldsymbol\theta}$--torsor $(\ce,\,
\boldsymbol\theta)$, where $\boldsymbol\theta$ is a system of real weights. We define
a parabolic Lie bracket operation on $\ce(\mathfrak g)$ as follows. The Lie bracket
$[,]\,:\, \gfr \otimes \gfr\,\lra\, \gfr$ is a $G$--equivariant homomorphism for the
adjoint action of $G$, and hence by \eqref{ass-con-real} we have a homomorphism of
parabolic vector bundles
$$\cE_{[,]}\,:\, \cE(\gfr \otimes \gfr)\,\lra\, \cE(\gfr)\, .$$

Now Proposition \ref{canisoparten} gives an isomorphism $$\cE(\gfr) \otimes^p \cE(\gfr)
\,\stackrel{\sim}{\longrightarrow}\, \cE(\gfr \otimes \gfr)$$ of parabolic vector
bundles. Combining with $\cE_{[,]}$, the Lie bracket can be defined as the parabolic homomorphism 
\begin{equation}\label{parhom}
[.,\, .]\,:\, \cE(\gfr) \otimes^p_X \cE(\gfr)\,\lra\, \cE(\gfr)\, .
\end{equation}

\section{Semistability and stability of torsors}\label{ss-s}

Let ${\mathcal G}_{_{\Omega,X}}$ be a Bruhat--Tits group scheme on $X$ as in 
Definition \ref{ht1}, and let $\boldsymbol\theta$ be such that ${\mathcal 
G}_{_{\Omega,X}} \,\simeq\, {{\mathcal G}_{_{{\boldsymbol\theta},X}}}$. Let 
$(\ce,\,\boldsymbol\theta)$ be a parahoric ${{\mathcal 
G}_{_{{\boldsymbol\theta},X}}}$--torsor with arbitrary real weights
$\boldsymbol\theta \,\in\, \ca_{_T}^{^m}$.

\begin{remark}\label{sgaremark}
Let $G_{_K}$ be a split group scheme over a field $K$. Let $E_{_K}$ be a 
$G_{_K}$--torsor. Consider the twisted group scheme $E(G)_{_K}$. Then
giving a parabolic subgroup scheme $$P_{_K} \,\subset\, E(G)_{_K}$$ (of fiber type 
$P$) is equivalent to giving a reduction of structure group of $E_{_K}$ to $P$ 
(cf. \cite[Expos\'e XXVI, Cor. 3.6]{sga}).
\end{remark}

\subsection{First definition}\label{fc} 

Let $\ce(\mathcal{G})$ denote the group scheme of automorphisms of $\ce$ obtained by
taking the quotient $\cE \times_X \cG$ by the left $\cg$--action on $\cE$ and the right
$\cG$--action on itself by conjugation. Let ${\rm Lie}(\ce(\mathcal G))$ denote the Lie algebra bundle of $\ce(\mathcal{G})$. One has the following well-known identification of Lie algebra bundles:
\begin{equation}\label{siso}
{\rm Lie}(\ce(\cG)) = \cE(\gfr) \,.
\end{equation}
Since $(\ce,\,\boldsymbol\theta)$ is a parahoric torsor, by Section \ref{lub}, the
above Lie algebra bundle $\cE(\gfr)$ gets a natural parabolic Lie algebra bundle
structure $(\cE(\gfr),\,\theta_\gfr)$. Thus via the isomorphism \eqref{siso} the
bundle ${\rm Lie}(\cE(\cG))$ gets a parabolic vector bundle structure with a Lie bracket
operation compatible with the parabolic structure:
\begin{equation}
({\rm Lie}(\cE(\cG))\, ,\,\theta_\gfr\, ,\,[.\, ,.\,])\,.
\end{equation}
By Remark \ref{sgaremark}, giving a generic parabolic reduction is equivalent to giving a parabolic subgroup scheme $$P_{_K} \subset \ce({\mathcal G}_{_K}).$$

For each $1\, \leq\, j\, \leq\, m$, we take the flat closure of $P_{_K}$ in 
$\ce({\mathcal G}_{_{A_{_j}}}).$ This will give a subgroup scheme $\eP \,\subset\,
\ce({\mathcal G}).$

The extended subgroup scheme $\eP$ also gives a Lie subalgebra bundle $${\rm Lie}(\eP) 
\,\subset\, {\rm Lie}(\cE(\cG)) \,\simeq\, \cE(\gfr)\, .$$ Then endow the bundle ${\rm Lie}(\eP)$ 
with the {\em canonical induced parabolic structure} on the divisor $D$ and 
denote this parabolic subbundle by ${\rm Lie}(\eP)_{_*}$.

\begin{definition}\label{newss}
We say that the parahoric torsor $(\ce,\,\boldsymbol\theta)$ with arbitrary real weights
$\boldsymbol\theta \,\in\, \ca_{_T}^{^m}$ is semistable 
(respectively, stable) as a ${\mathcal G}_{_{\boldsymbol\theta,X}}$--torsor if for
every generic parabolic reduction datum as above,
$$
{\rm par.deg}({\rm Lie}(\eP)_{_*}) \,\leq\, 0\ \ {\rm (respectively,}\ \ 
{\rm par.deg}({\rm Lie}(\eP)_{_*})\,<\, 0{\rm )}\, .
$$
\end{definition}

\subsection{Second definition}\label{2construction}

Now assume that $(\ce,\,\boldsymbol\theta)$ is a parahoric torsor with {\em rational
weights}. Let ${P}_K\,\subset\, {\mathcal G}_K$ be a maximal parabolic subgroup of the
generic fiber ${\mathcal G}_K$ of ${\mathcal G}_{_{\Omega,X}}$. Let $\chi\,:\, { P}_K \,
\longrightarrow\,{\mathbb G}_{m,K}$ be a strictly anti-dominant character of the
parabolic subgroup ${ P}_K$. Therefore, the associated line bundle on ${\mathcal
G}_K/{ P}_K$ is ample. Since the quotient map ${\eE}_K\, \longrightarrow\,
{\eE}_K/{ P}_K$ defines a principal ${ P}_K$--bundle, it follows that
$\chi$ defines a line bundle $L_{_\chi}$ on ${\eE}_K/{ P}_K\,=\,
{\eE}_K({\mathcal G}_K/{ P}_K)$. For any reduction of structure group
$$
s_K\, :\, X\setminus D\, \longrightarrow\, {\eE}_K({\mathcal G}_K/{ P}_K)\, ,
$$
we have the pulled back line bundle $s_K^{*}(L_{_\chi})$ on $X\setminus D$.
This line bundle $s_K^{*}(L_{_\chi})$ extends $X$ by the following 
proposition.

\begin{proposition}[{\cite[Proposition 6.3.1]{bs}}]\label{better} 
Let $(\eE,\,\boldsymbol\theta)$ be a parahoric torsor with rational weights. Let $s_K$ be a generic reduction of
structure group of ${\eE}_K$ to ${ P}_K$. Then the line bundle $s_K^{*}(L_{_\chi})$ on
$X\setminus D$ has a canonical extension $L_{_\chi}^{\boldsymbol\theta}$ to $X$ as a parabolic line bundle.
\end{proposition}

\begin{definition}[{\cite[Definition 6.3.4]{bs}}]\label{stability}
A parahoric torsor $(\eE,\, {\boldsymbol\theta})$ with rational weights is called {\em
stable (respectively, semistable)} if for every maximal parabolic ${\mathcal P}_K \,
\subset\, {\mathcal G}_K$, for all strictly anti-dominant character $\chi$ of
${\mathcal P}_K$, and for every reduction of structure group $s_K$ as above,
\[
{\rm par.deg}(L_{_\chi}^{\boldsymbol\theta}) \,<\, 0\ \ {\rm (respectively,}\ 
{\rm par.deg}(L_{_\chi}^{\boldsymbol\theta})
\,\leq\, 0{\rm )}\, .
\]
\end{definition}

We observe that the two definitions, namely Definition \ref{newss} and Definition 
\ref{stability}, are equivalent when the weights are rational; its proof is identical to 
the proof of \cite[Lemma 2.1]{Ra}.

\begin{remark}
In a recent paper, Heinloth (\cite{He2}) studies the Hilbert-Mumford criterion in terms 
of algebraic stacks. The point of view developed there allows natural choices of 
test objects for the verification of stability, leading to criteria for the existence of 
separated coarse moduli spaces.

Let $\mathcal E$ is a $\mathcal G$-torsor for a parahoric group scheme $\mathcal G$ and 
let $\mathcal B \,\subset\, Aut_{_{\mathcal G}}(\mathcal E)$. Let $\mathcal E_{_{\mathcal 
B}}$ be a reduction of structure group to $\mathcal B$. To a character $\chi$ of 
$\mathcal B$, one can then associate a line bundle $\mathcal E_{_{\mathcal B}}(\lambda)$ 
on $X$. In the language of Section \ref{2construction} this will be a parabolic line 
bundle. In \cite[Section 3.5]{He2} the classical notion of a parabolic degree of 
$\mathcal E_{_{\mathcal B}}(\lambda)$ is recovered in the new setting which leads to the 
definition of stability as in Section \ref{2construction}. An advantage with his 
new definition is that it works in positive characteristics as well and moreover, the 
parahoric group scheme $\mathcal G$ need not be assumed to be generically split as is 
done in the present paper.
\end{remark}

\section{(Semi)stability and polystability under variation of 
weights}\label{variation-ss}

Let $\eV$ be the category of all quasi--parahoric ${\mathcal 
G}_{_{\boldsymbol\theta,X}}$--torsors along $D$ with weights $\boldsymbol\theta$ varying 
in $\ca_{_T}^{^m}$. For convenience, we shall work with a {\it single} parabolic point $P 
\,\in\, X$. The generalization to finitely many points follows without any difficulty.

Let $r$ be the rank of $\cE(\gfr)_*$. The degree of the underlying vector bundle is 
denoted by $d_1$. Note that its parabolic degree is $0$ because $\gfr\,=\, 
{\gfr}^*$ as $G$--modules. Let $\eV_\gfr$ denote the space of all parabolic 
vector bundles such that
\begin{itemize}
\item the rank is the fixed integer $r$,

\item the quasi--parabolic structure at $P$ is 
given by that of $\cE(\gfr)_*$,

\item the degree of the underlying vector bundle is $d_1$,

\item the parabolic degree is zero, and

\item the parabolic weights $0 \,\leq\, \alpha_1 \,< \,\alpha_2 \,<\, \cdots 
\,<\, \alpha_{r'} \,<\,1$ are not fixed, but the length $r'$ and the multiplicities
$m_1,\, \cdots,\, m_{r'}$ are fixed.
\end{itemize}
 
Take any $V_* \,\in\, \eV_\gfr$. For a subbundle $W$ of $V$, if $n_{k}$ for each
$1 \,\leq\, k \,\leq\, r'$ denote the multiplicity of $\alpha_{k}$ for the
parabolic structure on $W$ induced by $\eV_\gfr$, the condition for
$\alpha_*$--stability of $V_*$ is
\begin{equation}\label{scondform}
{\rm degree}(W)+ \sum_{k \leq r'} n_{i_k} \alpha_{i_k} \, <\, 0
\end{equation}
for all subbundles $W$ of $V$; for semistability the strict inequality is replaced by
inequality. Let $\chi(V,W,\alpha)$ denote the left hand side of \eqref{scondform}. 

There exists a constant $C_2 \,\geq\, 0$ such that if ${\rm degree}(W) \,>\, C_2$,
then $\chi(V,W,\alpha) \,\geq\, 0$ for all $\alpha$. Thus for any $V \,\in\, \eV_\gfr$,
if the underlying vector bundle $V$ admits a subbundle $W$ with
${\rm degree}(W) \,\geq\, C_2$, then
$V_*$ can never be parabolic semistable for any choice of weights $\alpha_*$. We note
that the quasi--parabolic structure of $V \,\in\, \eV_\gfr$ alone determines such bundles.

\subsection{Facets of a quasi--parahoric torsor}

As before, $G$ is simple and simply connected.

In this subsection we shall only consider parahoric torsors $\cE \,\in\, \eV$ such that the
quasi--parabolic bundle $\cE(\gfr)$ admits no subbundle of degree greater than
$C_2$:
\begin{equation}\label{C2}
\eV_\gfr(C_2)\,:=\, \{ V_* \,\in\, \eV_\gfr\,\mid\, V~ \text{ has no subbundle of degree
greater than }~ C_2 \}\, .
\end{equation}

\begin{proposition}\label{finitesetofineq}
The set of inequalities required to verify the (semi)stability of any bundle in 
$\eV_\gfr(C_2)$ has finite cardinality.
\end{proposition}

\begin{proof}
{}From \eqref{scondform} we see that there exists a constant $C_1 \leq 0$ such that
$${\rm degree}(W) \,\leq\, d_1+ C_1\ \ \implies\ \chi(V,W,\alpha) \,<\,0\, .$$ In other words, subbundles of degree at most $C_1$ will never be destabilizing with respect to any inequality. 
Thus to check (semi)stability of $(\cE(\gfr),\,\theta_\gfr)$ we may restrict ourselves to subbundles $W$ of $\cE(\gfr)$ such that 
\begin{equation}
C_1 \,\leq \, {\rm degree}(W) \, \leq C_2 \,.
\end{equation}

The ranks of subbundles $W$ vary between $1$ and $r-1$. Let $m_1,\,\cdots,\, m_{r'}$
be the multiplicities of the quasi--parabolic structure on $\eV_\gfr$. The
multiplicities $n_{i_1},\, \cdots,\, n_{i_k}$ of the induced parabolic structure are
positive integers. Thus as one varies over $\eV_\gfr$, only finitely many inequalities appear. 
\end{proof}

If $\cE(\gfr) \,\in\, \eV_\gfr(C_2)$, it follows that to check (semi)stability of 
$(\ce,\,\boldsymbol\theta)$ we need to consider only finitely many inequalities 
corresponding to a (possibly proper) subset of the set of inequalities seen in 
Proposition \ref{finitesetofineq}. This is because we need to check these inequalities 
for subbundles which are Lie algebra bundles of certain subgroup schemes (see Definition 
\ref{newss}).

We fix a maximal torus $T$ of $G$ and also fix a maximal torus $T_\gfr$ of 
$GL(\gfr)$ such that $Ad(T_G) \,\subset\, T_\gfr$.

For every inequality, $${\rm degree}(W) + \sum_{k \leq r'} n_{i_k}
\alpha_{i_k} < 0\ \ 
{\rm (respectively,}\, \leq\, 0{\rm )}$$ in (\ref{scondform}) and
for every integer $c$ between $C_1$ (as in the proof of Proposition
\ref{finitesetofineq}) and $C_2$, we associate a functional
$$\ell_c\,:\, \ca_{_{T_\gfr}} \,\longrightarrow\, \mathbb{R}$$ as follows: for any
$\eta \,=\, (\alpha_1,\, \cdots,\, \alpha_r)\,\in\, \ca_{_{T_\gfr}}$,
$$\ell_c(\eta) \,=\, \frac{c}{{\rm rank}(W)} + \frac{\sum_{k \leq r'} n_{i_k}
\alpha_{i_k}}{{\rm
rank}(W)}.$$
Define $f\,:\,\ca_{_T} \,\longrightarrow\, \mathbb R$ by $f(\theta)
\,= \,\ell_c(\theta(\mathfrak g))$. These are finitely many in number. We denote the set of these functionals by $S_T^\cE$ (or $S_T$ for notational
convenience). Further they are defined over rationals, since clearly the definition of $f$ only involves rational numbers and the map $\ca_{_T} \,\lra\, \ca_{_{T_\gfr}}$ is defined over rationals and is linear.
For any functional in $f$ in $S_T$, define the $f$--wall in $\cA_{_T}$ as $$W_f\,
:=\, \{x \,\in\, \cA_{_T}\,\mid\, f(x)\,=\,0 \}\, .$$ The collection
$\{W_f\}_{f \in S_T^\cE}$ will be called the \textit{walls} of $\cE$.

\begin{definition}
Fix a quasi--parahoric torsor $\cE$. For any $\theta\,\in\, \ca_{_T}$, let $$S_1^\theta
\,=\, \{ f \,\in\, S_T^\cE\,\mid\, x \,\in\, W_f \}\, .$$ Let
$\mathcal{H}_n\,= \, \{ x \,\in\, \cA_{_T}\,\mid\, \quad |S_1^\theta| \,=\,n \}$.
Define a {\it facet} of $\cE$ to be a connected component of $\mathcal{H}_n$ for some
$n \,\geq\, 0$.
\end{definition}
Thus the facets of $\cE$ provide a decomposition 
$$\cA_{_T}\,=\, \bigsqcup_n \mathcal{H}_n$$
of $\cA_{_T}$. Note that only finitely many $n$ appear here.

For any weight $\theta \,\in\, \cA_{_T}$ there is a unique $n \,\geq\, 0$ such that
$\theta \,\in\, \mathcal{H}_n$. Thus for any $\theta \,\in\, \cA_{_T}$, the
following three are equivalent:
\begin{enumerate}
\item $\theta$ does not belong to any $\cE$--wall,

\item $S_1^\theta$ is empty, and

\item $\theta \,\in \,\mathcal{H}_0$.
\end{enumerate}

The \textit{facet} of $\theta$ is the unique facet of $\cE$ containing it.

The following propositions generalize \cite[Section 2, page 217]{MS}. We note that the 
quasi--parabolic bundles $\cE(\gfr)$ cannot admit any subbundle of degree greater than 
$C_2$ (cf. (\ref{C2})).

\begin{proposition} \label{Cor 2.9}
Let $\theta$ lie in $\cA_{_T}\,=\,Y(T) \otimes \RR$. Then there exists an element
$\theta'$ in the rational apartment $\ca_\QQ \,= \,Y(T) \otimes \QQ
$ such that for all $\ce \,\in \,\eV$, the pair
$(\ce,\,\theta)$ is semistable (respectively, stable) if and only
if $(\ce,\, \theta')$ is semistable (respectively, stable).
\end{proposition}

\begin{proof} 
Let $S_2^\theta$ be the complement of $S_1^\theta$ in $S_T$.

We note that $\theta$ has a strictly positive distance from each $W_f$, where $f
\,\in\, S_2^\theta$. Let $d$ be the minimum distance if $S_2^\theta$ is nonempty
and set $d$ to be $\infty$ if $S_2^\theta$ is empty. Thus in all cases $d \,>\, 0$. Let
$U$ be the ball in the alcove of radius $d$ around $\theta$. 

Let $I$ denote the $\cE$--facet of $\theta$. Let $I_1$ be connected component of 
$I \cap U$ containing $\theta$, where $U$ is the above ball. Now if $I_1$ is not
reduced to a single point, then we 
can take a rational weight $\theta'$ in it. If $I_1$ is just a point, then $\theta$ 
must be rational because $d\,>\,0$ and all the functionals are defined over 
$\mathbb{Q}$. In this case, we take $\theta'$ to be $\theta$ itself.

Let us check that the (semi)stability conditions for $\theta$ and $\theta'$ coincide. For each functional $J$ in $S_T$, 
\begin{enumerate}
\item if $J \,\in\, S_1^\theta$ then $J(\theta')\,=\,J(\theta)$, because
$\theta' \,\in\, I_1 \,\subset\, I$, and
\item[2.] if $J \,\in \,S_2$ then ${\rm sign}(J(\theta')) \,=\,
{\rm sign}(J(\theta))$, because $\theta' \in I_1 \subset U$.
\end{enumerate}
So for $\cE$, one has $\theta'$--(semi)stability is equivalent
to $\theta$--(semi)stability.
\end{proof}

We return to the setting of $ m$--marked points on $X$ noting that the above 
discussion immediately goes through for multiple marked points.

\begin{lemma}\label{fin-rep-stab}
Let $\boldsymbol\theta \,\in\, \ca_{_T}^{^m}$. Let $\rho_i\,:\, G \,\lra\,
{\rm GL}(V_i)$ for $i \,\leq\, m$ be finitely many representations. Then there exists
$\boldsymbol\theta' \,\in\, \ca_{\bq}^{^m}$ such that for any 
$\ce \,\in \,\eV$, and any $i \,\leq \,m$, the parahoric torsor
$(\ce,\,\boldsymbol\theta(V_i))$ is stable (respectively, semistable) if and only if
$(\ce,\, \boldsymbol\theta'(V_i))$ is stable (respectively, semistable).
\end{lemma}

\begin{proof}
This is immediate from Proposition \ref{Cor 2.9}.
\end{proof}

The following proposition is a generalization of \cite[Proposition 2.1]{MS}.

\begin{proposition}\label{Prop 2.1}
Given any $\boldsymbol\theta_{_0} \,\in\, \cA_{_T}^{^m}$, there exists a neighborhood
$U$ of $\boldsymbol\theta_{_0}$ in $\cA_{_T}^{^m}$ with the property
that for all $\ce \,\in\, \eV$ such that $(\ce,\,\boldsymbol\theta_{_0})$ is stable,
the pair $(\ce,\, \boldsymbol\theta)$ is stable for all $\boldsymbol\theta \,\in\, U$.
\end{proposition}

\begin{proof}
Now $\theta$ may as well be rational. Owing to the stability condition, $f(\theta)
\,<\, 0$ for all $f \,\in\, S_T$. Thus we have $S_2^\theta\,=\,S_T$. Let $d$ be the
minimum distance between 
$\theta$ and any $f$--wall. Now we take $U$ to be the ball around $\theta$ radius $d$.
\end{proof}

\begin{definition}\label{polystable-pvb} 
A parahoric torsor $(\ce,\,\boldsymbol\theta)$ for the linear parahoric group scheme 
$\mathcal{GL}(V)$ is called {\em polystable} if the associated parabolic vector 
bundle $\ce(V)_{_*}$ is polystable (i.e., a direct sum of stable parabolic bundles 
of parabolic degree $0$).
\end{definition}

It is straight-forward to check that Lemma \ref{fin-rep-stab} remains valid
if stability in the lemma is substituted by polystability.

\begin{corollary}\label{stabletopolystable} Let $\rho\,:\,G \,\longrightarrow\,
{\rm GL}(V)$ be a representation and $\rho_{_X}\,:
\,{\mathcal G} \,\longrightarrow\, \mathcal{GL}(V)$ the induced
homomorphism of parahoric group schemes as in 
\eqref{assocparabstructure}. Let $\ce$ be a $\cg$--torsor. Then, for a weight 
$\boldsymbol\theta \,\in\, \ca_{_T}^{^m}$ such that 
$(\ce,\, \boldsymbol\theta)$ is stable, the pair $(\ce(V)_{_*},\,
\boldsymbol\theta(V))$ is polystable.
\end{corollary}

\begin{proof} 
By Proposition (\ref{Prop 2.1}), we can assume $\theta$ is rational. For a stable 
equivariant principal $G$--bundle, the associated bundles are polystable. Consequently, 
in view of the equivalence of categories in Section \ref{mtbs}, The stability of $(\ce, 
\boldsymbol \theta)$ implies that $(\ce(V),\, \boldsymbol \theta(V)$ is polystable.
\end{proof}

\section{Connections on parahoric $\cG$--torsors}\label{con-defn}

The main aim of this section is to define connections on a parahoric $
\cg_{_{\boldsymbol\theta,X}}$--torsor.
 
\subsection{$D_X$--modules} We briefly recall the definition of $D_X$--modules.

\begin{definition}\label{dxmodule} Let $X \,\lra\, S$ be a $S$--scheme. Let $dx$ denote the image of $x$ under the canonical de Rham differentiation map $d\,:\,\cO_X
\,\longrightarrow\, \Omega^1_{X/S}$. Let $\cF$ be a coherent
sheaf of $\cO_X$--modules over $X$. By a $D_X$--\textit{module structure} on $\cF$ we mean a $\cO_S$--linear homomorphism of sheaf of 
abelian 
groups $\nabla\,:\, \cF \,\lra\, \cF \otimes_{\cO_X} \Omega^1_{X/S}$ satisfying
Leibniz rule which says that
\begin{equation}\label{Leib}
\nabla(x f)\,=\, 
f \otimes dx + x \nabla(f)\, ,
\end{equation}
\end{definition}
where $f$ and $x$ are local sections of $\cF$ and $\cO_X$ respectively.

\begin{definition} Let $\nabla_\cF\,:\, \cF\,\lra\, \cF \otimes \Omega^1_{X/S}$ and
$\nabla_\cE\,:\, \cE\,\lra\, \cE \otimes \Omega^1_{X/S}$ be two connections over
$\cF$ and $\cE$ respectively. Define their tensor product $\nabla_\cF \otimes
\nabla_\cE\,:\, \cF \otimes \cE \,\lra\, \cF \otimes \cE \otimes \Omega^1_{X/S}$
to be 
\begin{equation}\label{tensorp}
\nabla_{\cF \otimes \cE} (f \otimes e)\,=\, \nabla_\cF(f) \otimes e + 
f \otimes \nabla_\cE(e)\, , 
\end{equation}
where $f$ and $e$ are local sections of $\cF$ and $\cE$ respectively. 

Similarly define $\nabla_\Hom: Hom(\cE,\cF) \,\lra \,Hom(\cE,\cF) \otimes
\Omega^1_{X/S}$ to be 
\begin{equation}\label{homcon} \nabla_{\Hom(\cE,\cF)}(\Phi)(e)\,=\,
\nabla_\cF(\Phi(e)) - \Phi(\nabla_\cE(e))\, , \end{equation}
where $\Phi$ and $e$ are local sections of $Hom(\cE,\cF)$ and $\cE$
respectively.
\end{definition}

\subsection{Logarithmic connections on curves}

The canonical line bundle of the smooth complex projective curve $X$ will be denoted
by $K_X$. Fix a finite subset $D\,=\, \{x_i\}_{1 \leq i \leq m}\, \subset\, X$; define
$$K_X(\log D)\,=\,K_X\otimes {\mathcal O}_X(D)\, .$$
A logarithmic connection on a vector bundle $V\,\longrightarrow\, X$ singular on $D$
is a first order algebraic differential operator $\nabla\,:\, V \,\lra\,
V \otimes K_X(\log D)$ satisfying the Leibniz rule.

For a point $x\,\in\, D$, the fiber $K_X(\log D)_x$ is identified with $\CC$ using the
Poincar\'e adjunction formula. For a logarithmic connection $(V,\, \nabla)$, the composition
$$ V \,\stackrel{\nabla}{\lra}\, V \otimes K_X(\log D)\, \lra \,(V \otimes K_X(\log D))_{x}
\,\lra\, V_x\, , $$
which is a $\mathbb C$--linear endomorphism of $V_x$,
is called the {\it residue of $\nabla$ at $x$} \cite[page 53]{del}, and it
is denoted by ${\rm Res}(\nabla,x)$. The monodromy of $\nabla$ around $x$ is conjugate to
\begin{equation}\label{mr}
\exp(-2\pi\sqrt{-1}{\rm Res}(\nabla,x))
\end{equation}
\cite[p.~53, Th\'eor\`eme~1.17]{del}.

\subsection{Connection on parabolic vector bundles}

Let $V\,\lra\, X$ be a vector bundle on $X$. A {\it quasi--parabolic} structure on $V$
over $D$ is a filtration, for each $x \,\in \, D$, of subspaces
\begin{equation}\label{flag}
V_x \,=\, F^x_1 \,\supsetneq \, F^x_2 \,\supsetneq\, \cdots \,\supsetneq \,F^x_{a_x}
\,\supsetneq\, F^x_{a_x +1} \,=\, \{0\}\, .
\end{equation}
A {\it parabolic structure} on $V$ over $D$ is a quasi--parabolic structure as above
together with real numbers 
$$
0 \,\leq\, \alpha^x_1 \,< \,\cdots \,<\, \alpha^x_i \,<\,\cdots\, < \alpha^x_{a_x} \, <\,1
$$
associated to the quasi--parabolic flags.
We shall often abbreviate a parabolic vector bundle
$(V,\,\{F^x_* ,\, \alpha^x_*\}_{x \in D})$ by $V_*$. 

\begin{definition}\label{con-on-parvb} A connection on $V_*$ is a logarithmic connection
$\nabla$ on $V$ such that for all $x \,\in\, D$,
\begin{itemize}
\item the residue ${\rm Res}(\nabla,x)$ is semisimple and preserves the quasi--parabolic
flag at $x$, meaning ${\rm Res}(\nabla,x)(F^x_i) \,\subseteq\, F^x_i$ for all $i$, and 

\item ${\rm Res}(\nabla,x)(F^x_i/F^x_{i+1}) \,= \, \alpha^x_i \Id_{F^x_i/F^x_{i+1}}$.
\end{itemize}
\end{definition}

A connection on $V_*$ induces a connection on the dual parabolic vector bundle $V^*_*$. 
To see this, given a logarithmic connection $\nabla$ on $V$ defining a connection on 
$V_*$, consider the logarithmic connection on $V^*\otimes {\mathcal O}_X(D)$ induced by 
$\nabla$. This logarithmic connection preserves the subsheaf of $V^*\otimes {\mathcal 
O}_X(D)$ identified with the vector bundle underlying the parabolic vector bundle 
$V^*_*$. The logarithmic connection on this subsheaf obtained this way defines a 
connection on $V^*_*$.

Let $V^1_*$ and $V^2_*$ be parabolic vector bundles with underlying vector bundles $V^1$ 
and $V^2$ respectively. Let $\nabla^1$ and $\nabla^2$ be connections on $V^1_*$ and 
$V^2_*$ respectively. Consider the logarithmic connection on $V^1\otimes V^2\otimes 
{\mathcal O}_X(D)$ induced by $\nabla^1$ and $\nabla^2$. It preserves the subsheaf of 
$V^1\otimes V^2\otimes {\mathcal O}_X(D)$ corresponding to the parabolic tensor product 
$V^1_*\otimes^p V^2_*$. The logarithmic connection on this subsheaf obtained this way 
defines a connection on $V^1_*\otimes^p V^2_*$.

\subsection{Lie connection on a principal $G$--bundle}

For a principal $G$--bundle $E
\,\lra\, X$, let $E(\gfr)\,=\, E_G\times^G \gfr$ be its adjoint bundle. The fibers of
$E(\gfr)$ are equipped with a Lie bracket structure $[.,.]\,:\,
 E(\gfr) \otimes E(\gfr)\,\lra\, E(\gfr)$ induced by the Lie algebra structure of
$\gfr$.

\begin{definition}\label{Lieconprin}
A {\it Lie connection} on $E$ is a connection $$\nabla\,:\, E(\gfr)\,\lra\, E(\gfr) \otimes
\Omega^1_X$$ such that following diagram is commutative
\begin{equation}\label{Lieconprinseq}
\xymatrix{
E(\gfr) \otimes E(\gfr) \ar[rrr]^{[.,.]} \ar[d]_{\nabla_\otimes} &&& E(\gfr) \ar[d]^{\nabla} \\
E(\gfr) \otimes E(\gfr) \otimes \Omega^1_{X} \ar[rrr]^{~~~~~[.,.] \otimes \Id_{\Omega^1_{X}}} &&& E(\gfr) \otimes \Omega^1_{X}
}
\end{equation}
where $\nabla_\otimes$ is the connection on $E(\gfr) \otimes E(\gfr)$
induced by $\nabla$.
\end{definition}

The above commutativity condition means that the section of $E(\gfr)\otimes (E(\gfr) 
\otimes E(\gfr))^*$ given by the Lie bracket operation on $E(\gfr)$ is flat with respect 
to the connection on $E(\gfr)\otimes (E(\gfr) \otimes E(\gfr))^*$ induced by $\nabla$.

A connection on $E_G$ produces a Lie connection on $E_G$. Therefore, we get a map from 
the space of all connections on $E_G$ to the space of all Lie connection on $E_G$. Since 
$G$ is semisimple, the adjoint homomorphism $G\, \longrightarrow\, \text{GL}(\gfr)$ has 
finite kernel and its image is the connected component, containing the identity element, 
of the group of all automorphisms of the Lie algebra $\gfr$. From this it follows that 
the above map from the space of all connections on $E_G$ to the space of all Lie 
connection on $E_G$ is a bijection. This fact motivates the definition of a connection on 
a parahoric torsor.

\subsection{Connection on parahoric $\cG$--torsors}\label{con-on-tor}

We refer to Section \ref{lub} for the notation. Take any $\cG\,=\, \cg_{_{\boldsymbol\theta,X}}$.

\begin{definition}\label{Lieconpar}
A {\it connection} on a $\cG$--torsor $\cE$ is a logarithmic parabolic connection (see Definition
\ref{con-on-parvb}) $\nabla$ on $\cE(\gfr)$ satisfying the condition that
the section of $Hom^p(\cE(\gfr) \otimes^p \cE(\gfr), \, \cE(\gfr))$ given by the homomorphism in
\eqref{parhom} is flat with respect to the connection $\nabla_{hom}$
on the parabolic vector bundle $Hom^p(\cE(\gfr) \otimes^p \cE(\gfr), \cE(\gfr))$
induced by $\nabla$.
\end{definition}

\subsection{A Tannakian description of connection}

Let $M$ be a smooth complex variety. Let $G$ be a complex reductive algebraic group and
$E_G\, \longrightarrow\, M$ a principal $G$--bundle. Take any pair $(H,\, f)$, where
$H$ is a complex algebraic group and $f\, :\, G\, \longrightarrow\, H$ an algebraic
homomorphism such that corresponding homomorphism of Lie algebras
$df\, :\, \text{Lie}(G)\, \longrightarrow\, \text{Lie}(H)$ is injective.
Let $E_H\, :=\, E_H\times^f H\, \longrightarrow\, M$ be the principal $H$--bundle obtained
by extending the structure group of $E_G$ using $f$. Let $$\widetilde{f}\, :\,
E_G\, \longrightarrow\, E_H$$ be the natural morphism. A connection on $E_G$ induces a
connection on $E_H$. The converse is also true. To see this, fix a $G$--equivariant
splitting
$$
\sigma\, :\, \text{Lie}(H)\, \longrightarrow\, \text{Lie}(G)\, ,
$$
meaning $\sigma\circ df\,=\, \text{Id}_{\text{Lie}(G)}$ (such a splitting exists because 
$G$ is reductive). If $D$ is a $\text{Lie}(H)$--valued $1$--form on $E_H$ defining a 
connection on $H$, then $\sigma\circ \widetilde{f}^*D$ is a connection on $E_G$. If $D_0$ 
is a connection on $E_G$ and $D$ the connection on $E_H$ induced by $D_0$, then the 
connection $\sigma\circ \widetilde{f}^*D$ on $E_G$ coincides with $D_0$. Indeed, this
follows immediately from the fact that $\sigma\circ df\,=\, \text{Id}_{\text{Lie}(G)}$.

Therefore, the map from connections on $E_G$ to connections on $E_H$ is injective. The 
image of this map from connections on $E_G$ admits a group theoretic description. This
will be explained below.

We remove the assumption that the algebraic group $G$ is reductive. As before $f\, :\, G\, 
\longrightarrow\, H$ to be any algebraic homomorphism such that $df$ is injective.

A theorem of C. Chevalley (see \cite[p. 80]{Hu}) says that there is a finite dimensional 
left $H$--module
$$
\rho\, :\, H\, \longrightarrow\, \text{GL}(W)
$$
and a complex line $\ell\, \subset\, W$ such that $f(G)$ is exactly the isotropy 
subgroup, of the point in the projective space $P(W)$ representing the line $\ell$, for 
the action of $H$ on the projective space $P(W)$ of lines in $W$ induced by the action of 
$H$ on $W$. Let $$E_W\, :=\, E_H\times^H W\, \longrightarrow\,M$$ be the vector bundle 
associated to $E_H$ for the $H$--module $W$. For a connection $D$ on $E_H$, the 
connection on $E_W$ induced by $D$ will be denoted by $D_W$. Note that $E_W$ is 
identified with the vector bundle associated to $E_G$ for the action $\rho\circ f$ of $G$ 
on $W$. The condition on $\ell$ implies that the action of $G$ on $W$ preserves
it. Let $E_\ell\, \subset\, E_W$ be the line subbundle associated to the 
$G$--submodule $\ell\, \subset\, W$.

A connection $D$ on $E_H$ is induced by a connection on $E_G$ if and only if the 
corresponding connection $D_W$ on $E_W$ preserves the above line subbundle $E_\ell\, 
\subset\, E_W$. This characterizes the connections on $E_H$ that are induced by 
connections on $E_G$.

We recall that the Tannakian theory involves describing properties of principal
bundles in terms of properties of associated vector bundles. For a
Tannakian description of connections on $E_G$, take $H\,=\, \text{GL}(V)$, so
$V$ is a finite dimensional $G$--module. Let
$$
E_V\, :=\, E_G\times^f V\, \longrightarrow\,M
$$ 
be the vector bundle associated to $E_G$ for the $G$--module $V$. From the above
observation we know that a connection on $E_G$ is a connection $D$ on the vector bundle
$E_V$ such that the connection on the vector bundle $E_W$ induced by $D$ preserves
the line subbundle $E_\ell\, \subset\, E_W$.

\section{Connections on $(\Gamma,G)$--bundles and rational weights} 

Let $F$ be a principal $G$--bundle on a curve $Y$ with adjoint bundle
\begin{equation}\label{pi}
\text{Ad}(F)\,=\,
F(G)\,=\,F \times^G G\, ,
\end{equation}
where $G$ acts on itself by conjugation. Given a principal $G$--bundle $E$ on $Y$, define
the ``twisting'' by $F$
$$
E \times^G F^{op}\,:=\, (E\times_Y F)/\sim\, ,
$$
where the equivalence relation identifies all pairs $(e,\, f)\, ,\, (ez,\, fz)\, \in\,
E\times_Y F$, with $z\,\in\, G$. Consider the map
$$
\xi\, :\, E\times_Y F\times_Y F\times G \, \longrightarrow\, E\times_Y F\, ,\ \
(e,\, f,\, fz,\, z_1)\, \longmapsto\, (ezz_1,\, fz)\, ,
$$
where $(e,\, f)\, \in\, E\times_Y F$ and $z,\, z_1\, \in\, G$. There is a unique map
$$
\widehat{\xi}\, :\, (E \times^G F^{op})\times_Y F(G)\, \longrightarrow\, E \times^G F^{op}
$$
such that the following diagram is commutative
$$
\begin{matrix}
(E\times_Y F)\times_Y (F\times G) & \stackrel{\xi}{\longrightarrow} & E\times_Y F\\
\Big\downarrow && \Big\downarrow\\
(E \times^G F^{op})\times_Y F(G)& \stackrel{\widehat\xi}{\longrightarrow} &E \times^G F^{op}\, ;
\end{matrix}
$$
recall that $E \times^G F^{op}$ and $F(G)$ are quotients of $E\times_Y F$ and $F\times G$
respectively. This map $\widehat{\xi}$ makes $E \times^G F^{op}$ a
$F(G)$--torsor on $Y$. Let $\alpha_F$ be the map from the space of principal $G$--bundles
to the space of $F(G)$--torsors on $Y$ defined by 
$E \,\longmapsto\, E \times^G F^{op}$. This $\alpha_F$ is an equivalence of categories.

Consider the adjoint action of $G$ on ${\rm Lie}(G)\,=\, {\mathfrak g}$.
Let $\text{ad}(F)\,=\, F\times^G{\mathfrak g} \, \longrightarrow\, Y$ be the
associated adjoint vector bundle. We note that $\text{ad}(F)$ is the Lie algebra
bundle associated to the group scheme $F(G)$.

Let $\nabla_0$ be a connection $F$. Using $\nabla_0$ we will define connections on 
a $F(G)$--torsor.

The connection $\nabla_0$ induces a connection on every fiber bundle associated to
$F$. In particular, it produces a connection on $F(G)$; this connection on
$F(G)$ given by $\nabla_0$ will be denoted by $\nabla^G_0$. The kernel of the
differential $d\pi\, :\, TF(G)\, \longrightarrow\, \pi^*TY$ of the map $\pi$ in
\eqref{pi} is identified with $\pi^*\text{ad}(F)$. So the above connection
$\nabla^G_0$ on $F(G)$ gives a homomorphism
\begin{equation}\label{nG}
\nabla^G_0 \, :\, TF(G)\, \longrightarrow\,\pi^*\text{ad}(F)\, .
\end{equation}

Take any $F(G)$--torsor $\varphi\, :\, E\, \longrightarrow\, Y$. Consider the
action $E\times F(G)\, \longrightarrow\, E$ of $F(G)$ on $E$. Let
\begin{equation}\label{de}
\delta\, :\, TE\oplus\varphi^* TF(G) \, \longrightarrow\,TE
\end{equation}
be the differential of this map giving the action.
Consider the differential of $\varphi$
$$
d\varphi\, :\, TE\, \longrightarrow\, \varphi^*TY\, .
$$
Let
$$
T_{\varphi}\, :=\, \text{kernel}(d\varphi) \, \subset\, TE
$$
be the relative tangent bundle for the projection $\varphi$. The action of
$F(G)$ on $E$ identifies $T_{\varphi}$ with $\varphi^*\text{ad}(F)$.

A \textit{connection} on a $F(G)$--torsor $\varphi\, :\,E\, \longrightarrow\, Y$ is a
holomorphic homomorphism of vector bundles over $Y$
$$
\beta\,: \, TE \, \longrightarrow\, \varphi^*{\rm ad}(F)\,=\, {\rm ad}(\varphi^*F)
$$
such that
\begin{enumerate}
\item the restriction of $\beta$ to $T_{\varphi}$ coincides with the above
identification of $T_{\varphi}$ with $\varphi^*\text{ad}(F)$, and

\item for the homomorphism $\delta$ in \eqref{de},
$$
\delta(\text{kernel}(\beta)\oplus \varphi^*\text{kernel}(\nabla^G_0))\, \subset\,
\text{kernel}(\beta)\, ,
$$
where $\nabla^G_0$ is the homomorphism in \eqref{nG}.
\end{enumerate}

Note that the above definition of a connection on $E$ depends on $\nabla_0$.

If $F$ is the trivial principal $G$--bundle $Y\times G$, then $F(G) \,=\, Y\times 
G$, and a $F(G)$--torsor is in fact a principal $G$--bundle on $Y$. If we choose 
$\nabla_0$ to be the trivial connection on $Y\times G$, then connections on a 
$F(G)$--torsor are same as connections on the corresponding principal $G$--bundle.

The following lemma is straight-forward to check.

\begin{lemma}\label{lem-c}
Given a $F(G)$--torsor $\varphi\, :\,E\, \longrightarrow\, Y$, a homomorphism 
$$
\beta\,: \, TE \, \longrightarrow\, \varphi^*{\rm ad}(F)\,=\, {\rm ad}(\varphi^*F)
$$
defines a connection on $E$ if and only if
$$
\delta^*\beta\,=\, \beta\oplus \nabla^G_0
$$
on $TE\oplus\varphi^* TF(G)$, where $\delta$ is constructed in \eqref{de} and
$\nabla^G_0$ is the homomorphism in \eqref{nG}.
\end{lemma}

\begin{proposition}\label{twist-equi}
Twisting by $F$ defines an equivalence
between principal $G$--bundles equipped with a connection and
$F(G)$--torsors equipped with a connection. 
\end{proposition}

\begin{proof}
Let $E$ be a principal $G$--bundle on $Y$. Let $D$ be a connection on $E$. Consider $D$
as a $\mathfrak g$--valued $1$--form on $E$. Let $D'$ denote the $\mathfrak g$--valued
$1$--form on $F$ corresponding to the connection $\nabla_0$ on $F$. So $(D,\, D')$ is
a $\mathfrak g$--valued
$1$--form on the fiber product $E\times_Y F$. The pullback of $\text{ad}(F)$ to $F$ is
identified with the trivial vector bundle $F\times {\mathfrak g}\,\longrightarrow\, F$.
Therefore, $(D,\, D')$ defines a $1$--form with values in the pullback of $\text{ad}(F)$
to $E\times_Y F$. This form on $E\times_Y F$ descends to the quotient $F(G)$--torsor
$E\times^G F^{op}$ as a $1$--form with values in the pullback of ${\rm ad}(F)$
to $E\times^G F^{op}$. It is straight-forward to check
that this form defines a connection on the $F(G)$--torsor corresponding to $E$.

Conversely, let $\beta$ be a connection on a $F(G)$--torsor $\varphi\, :\,E\, 
\longrightarrow\, Y$. Consider the pullback $\beta'$ of $\beta$ to $E\times_Y F$ 
as a $1$--form with values in the pullback of $\text{ad}(F)$. As noted above, the 
pullback of $\text{ad}(F)$ to $E\times_Y F$ is identified with the trivial vector 
bundle with fiber $\mathfrak g$. So $\beta'$ is a $1$--form on $E\times_Y F$ with 
values in $\mathfrak g$. Let $D'$ be the pullback of the connection form $D$ to 
$E\times_Y F$. Then $\beta'-D'$ descends to $E$ by the projection $E\times_Y F\, 
\longrightarrow\, E$, and this descended form defines a connection on the principal
$G$--bundle $E$.

The above two constructions are evidently inverses of each other.
\end{proof}

Assume that $Y$ is equipped with the action of a finite group $\Gamma$. A
$\Gamma$--{\it connection} on a $(\Gamma,G)$--bundle $E$ on $Y$ is a connection on
$E$ which is preserved by the action of $\Gamma$.

\begin{proposition}\label{bijcon}
Let $E \,\lra\, 
Y$ be a $(\Gamma,G)$--bundle on some Galois cover $p\,:\,Y \,\lra\, X$
with Galois group $\Gamma$. Let $\cE$ be the parahoric torsor on $X$
with rational weights corresponding to $E$. Then there is a natural bijection between
the connections on $\cE$ and the $\Gamma$--connections on $E$.
\end{proposition}

\begin{proof}
This follows from the fact that the connections on a $\Gamma$--equivariant vector bundle
on $Y$ are in bijection with the connections on the corresponding parabolic vector
bundle on $X$.
\end{proof}

\section{Flat unitary connections on parahoric torsors and stability}\label{EH}

\subsection{Polystable parahoric torsors}

\begin{lemma}\label{holonomy}
Let $V_*$ be a polystable parabolic vector bundle of parabolic degree zero with 
real weights $\theta$. Then the parabolic vector bundle $(V_*)^{\otimes m} 
\otimes^p((V_*)^*)^{\otimes n}$ is also polystable.
\end{lemma}

\begin{proof}
A parabolic vector bundle of parabolic degree zero is polystable if and only if it is
given by a unitary representation of $\pi_1(X\setminus D)$, where $D$ is the parabolic
divisor \cite{MS}, \cite{biquard}. Since $V_*$ is polystable, it is given by
a representation $\rho$ of $\pi_1(X\setminus D)$. The parabolic vector bundle
$(V_*)^{\otimes m} \otimes^p((V_*)^*)^{\otimes n}$ is given by the representation
$\rho^{\otimes m}\otimes (\overline{\rho})^{\otimes n}$. This
implies that $(V_*)^{\otimes m} \otimes^p((V_*)^*)^{\otimes n}$ is polystable.
\end{proof}

\begin{corollary}\label{holonomy2}
Take $V_*$ as in Lemma \ref{holonomy}. Take any homomorphism 
$$
\rho\, : \,{\rm GL}(r, {\mathbb C}) \, \longrightarrow\, {\rm GL}(N, {\mathbb C})\, ,
$$
where $r$ is the rank of $V_*$. Let $W_*$ be the parabolic vector bundle associated
to $V_*$ for $\rho$. Then $W_*$ is also polystable.
\end{corollary}

\begin{proof}
Consider ${\mathbb C}^N$ as a ${\rm GL}(r, {\mathbb C})$--module using $\rho$ and
the standard representation of ${\rm GL}(N, {\mathbb C})$. This ${\rm GL}(r,
{\mathbb C})$--module ${\mathbb C}^N$ is a direct summand of a direct sum of
${\rm GL}(r, {\mathbb C})$--modules of the form
$({\mathbb C}^r)^{\otimes m_i} (({\mathbb C}^r)^*)^{\otimes n_i}$ \cite[p. 40, Proposition
3.1(a)]{DMOS}. Therefore, from Lemma \ref{holonomy} we conclude that
$W_*$ is a direct summand of a polystable parabolic vector bundle of parabolic
degree zero. Hence $W_*$ is polystable.
\end{proof}

We define polystability for parahoric torsors.

\begin{definition}\label{polystability} Let $\cg_{_{\boldsymbol\theta,X}} \,\lra\, X$ be a Bruhat--Tits group
scheme with generic fiber $G$. A parahoric $\cg_{_{\boldsymbol\theta,X}}$--torsor $\cE$ with real
weights $\boldsymbol\theta$ is said to be polystable if for every representation $\rho\,: \,G
\,\lra\, {\rm GL}(V)$, the corresponding parabolic vector
bundle $(\cE,\boldsymbol\theta(V))$ is polystable in the sense of Definition
\ref{polystable-pvb}.
\end{definition}

\subsection{Polystable parahoric torsors from unitary representations}\label{se10.2}

In this subsection we will first assume that $D\,=\, \{x\}$ is a single point.
The multi-point case is actually a straight-forward generalization.

The complement $X\setminus \{x\}$ would be denoted by $Y$. For a base point $y_0\,\in\, Y$,
set $\Gamma \,=\, \pi_{1}(Y, \,y_0)$. Choose an analytic disc $U \,\subset\, X$ around $x$
such that $y_0\,\in\, U$. The inclusion of $U \setminus \{x\}$ in $Y$ produces 
an inclusion $\pi_{1}(U \setminus \{x\},\, y_0) \,\hookrightarrow\, \Gamma$. Using the orientation
of $U \setminus \{x\}$, the group $\pi_{1}(U \setminus \{x\},\, y_0)$ gets identified
with $\mathbb Z$. The element of $\pi_{1}(U \setminus \{x\},\, y_0)$ corresponding
to $1\, \in\, \mathbb Z$ will be denoted by $\gamma$.

We now recall a description of the set of conjugacy classes in a compact semisimple and simply 
connected group in terms of the Weyl alcove (see \cite[page 9]{bs}). Let $K_G \,\subset\, G$ be a 
fixed maximal compact subgroup and $T$ a fixed maximal torus in $K_G$. The corresponding Weyl group
in the quotient $N_{G}(T)/T$.
The set of conjugacy classes of element in $K_G$ gets identified with the $T/W$ which is in fact the Weyl 
alcove because any element of $K_G$ is conjugate to an element in the
maximal torus up to an element of the Weyl group (cf. \cite[page 151]{morgan}). Given any
$t \,\in\, K_{G}$, let $\theta_{t}$ denote the point in the Weyl alcove corresponding to $t$.

Given any homomorphism
$$
\rho\, :\, \Gamma\,\longrightarrow\, K_G\, ,
$$
let $E_{\rho}$ be the flat principal $G$--bundle on $Y$ associated to it. To construct
$E_{\rho}$, let $(\widetilde{Y},\, \widetilde{y}_0)$ be the pointed universal cover of $Y$ corresponding
to the base point $y_0$; note that $\Gamma$ acts on $\widetilde Y$. Identify two points $(y_1,\, g_1),\, (y_2,\, g_2)\, \in\,
\widetilde{Y}\times G$ if there is an element $\gamma\, \in\, \Gamma$ such that
$(y_2,\, g_2)\,=\, (y_1\gamma,\, \rho(\gamma^{-1})g_2)$. The quotient of $\widetilde{Y}\times G$
is a principal $G$--bundle on $Y$, which is denoted by
$E_{\rho}$; the right translation action of $G$ on $\widetilde{Y}\times G$
produces an action of $G$ on $E_{\rho}$. The flat connection on the trivial principal $G$--bundle
$\widetilde{Y}\times G\,\longrightarrow\, \widetilde Y$ descends to a flat connection on
$E_{\rho}$. For any $h\, \in\, K_G$, the map
$$
\widetilde{Y}\times G \, \longrightarrow\, \widetilde{Y}\times G\, ,\ \ (y,\, z) \,\longmapsto\,
(y,\, \rho(h)z)
$$
descends to an isomorphism
\begin{equation}\label{isomviaconj}
E_{h\rho h^{-1}}\,\stackrel{\sim}{\longrightarrow}\, E_{_\rho}
\end{equation}
as flat principal $G$--bundles on $Y$.

Let
$$
t\,=\,\rho(\gamma) \,\in\, K_G
$$
be the image of $\gamma$. Since $h.\rho. h^{-1}(\gamma) \,=\, h th^{-1}$ for
all $h\, \in\, K_G$, there is the map
$$
\text{Hom}(\Gamma,\, K_G)/K_G \, \longrightarrow\, {\mathcal A}\, ,\ \ \
\rho\, \longmapsto\, \theta_t\, .
$$
After conjugating $\rho$ by an element of $K_G$, we may assume that $t$ belongs to a fixed maximal 
torus $T$ of $K_{G}$.

Let ${\textbf t} \,\in\, \text{Lie}(T)$ be such that
\begin{equation}\label{tt}
\exp(-2\pi\sqrt{-1}{\textbf t}) \,=\, t\, ,
\end{equation}
where $\exp$ denotes the exponential map on the Lie algebra $\text{Lie}(T)$. Consider
the trivial principal $G$--bundle $(U \setminus \{x\})\times G$ over
$U \setminus \{x\}$. The trivial connection on it given by this trivialization will
be denoted by $d_0$. On $(U \setminus \{x\})\times G$, we now have the flat connection
\begin{equation}\label{cond}
\widehat{d}\, =\, d_0+ \frac{{\textbf t}dz}{z}\, ,
\end{equation}
where $z$ is a holomorphic coordinate function on $U$ with $z(x)\,=\, 0$.

Restrict the representation $\rho$ to the subgroup $\pi_{1}(U \setminus \{x\},\, y_0)$. 
This produces a flat principal $G$--bundle $E_{\rho}(\infty) \,\longrightarrow\,
U \setminus \{x\}$. Note that
\begin{equation}\label{viapsi}
E_{_\rho}|_{U \setminus \{x\}} \,\simeq\, E_{_\rho}(\infty)
\end{equation} 
as flat principal $G$--bundles. Both the flat principal $G$--bundles in \eqref{viapsi}
are isomorphic to the flat principal $G$--bundle $((U \setminus \{x\})\times G,\,
\widehat{d})$ constructed in \eqref{cond}. This is because all of them have the
same residue, namely ${\textbf t}$, at $x$. Recall that the residue determines the
conjugacy class of the monodromy (see \eqref{mr}).

To the element $t \,\in\, K_{G} \cap T$, we have the associated conjugacy class 
$\theta_{t}\,\in\, \mathcal A$, and hence by Bruhat--Tits theory have a group scheme 
$\cG_{_{\theta_{t}}}$ on a formal neighborhood $\widehat{U} \,=\, \spec \bc[[t]]$ of $x$. 
This group scheme $\cG_{_{\theta_{t}}}$ is the trivial group scheme $\spec \bc((t))\times 
G$ over $\widehat{U} \setminus x \,=\, \spec \bc((t))$. Therefore, we can extend 
$\cG_{_{\theta_{t}}}$ uniquely to $X$ by setting it to be the trivial group scheme $ 
Y\times G$ over $Y$ \cite[Section 5.2, page 28]{bs}. We denote this group scheme on $X$ 
by $\cG_{t,X}$.

We observe that there is a morphism $\widehat{U} \,\hookrightarrow\, {U}$, where one 
identifies the formal power series ring with the completion of the convergent power 
series ring. Consider the trivial $\cG_{_{\theta_{_t}}}$--torsor on $\widehat{U}$ which 
we denote by $P_{_t}$. Note that the connection on $E_{_\rho}(\infty)$ restricts to a 
natural connection on $P_{_t}|_{\widehat{U} \setminus \{x\}}$.

We now patch together (using for example \cite[5.2.3]{bs}) the trivial
$\cG_{_{\theta_{t}}}$--torsor $P_{_t}$ on $\widehat{U}$ and the principal $G$--bundle
$E_\rho$ over $Y$ along the intersection $\widehat{U} \setminus \{x\} \,=\, \spec \bc((t))$ such that the patching
is connection preserving; as noted above, on $\spec \bc((t))$, both the
principal $G$--bundles are the trivial principal $G$--bundle $\spec \bc((t))\times G$
equipped with the connection $\widehat{d}$.

The above construction is summarized in the following proposition.

\begin{proposition}\label{cc}
Given any homomorphism $\rho\, :\, \Gamma\,\longrightarrow\, K_G$,
and any ${\textbf t} \,\in\, {\rm Lie}(T)$ satisfying \eqref{tt}, the flat principal $G$--bundle
on $Y$ has a canonical extension to a $\cG_{t,X}$--torsor over $X$.
\end{proposition}

It should be clarified that the $\cG_{t,X}$--torsor in Proposition \eqref{cc} 
depends on the choice of ${\textbf t}$ (a branch of the logarithm), while the 
isomorphism class of the Bruhat--Tits group scheme $\cG_{t,X}$ depends only on the 
conjugacy class $[t]$ of $t$. For $G\, =\, \text{GL}(r, {\mathbb C})$, if the
logarithm ${\textbf t}$ is chosen as done in \cite{MS} (meaning ${\textbf t}$ is
semisimple and eigenvalues are nonnegative and less than $1$), then the construction
in \eqref{cc} coincides with the construction in \cite{MS} of a parabolic vector
bundle from a homomorphism $\Gamma\,\longrightarrow\, \text{U}(r)$. This follows
by comparing the two constructions.

The $\cG_{t,X}$--torsor in Proposition \eqref{cc} will be denoted by
$E_{\rho}(t)$.

It may be mentioned that we may restrict the connection $\widehat{d}$ in 
\eqref{cond} to $\widehat{U}\setminus\{x\}\, \subset\, U\setminus \{x\}$, where 
$\widehat U$ as before is the formal completion along $x$. A $\cG_{t,X}$--torsor on 
$X$ can be trivialized over both $Y$ and $\widehat U$. Conversely, given 
$\cG_{t,X}$--torsors on $Y$ and $\widehat U$, and an isomorphism between them over 
$\widehat{U}\setminus\{x\}$, we get a $\cG_{t,X}$--torsor on $X$. Therefore, the 
connection over $\widehat{U}\setminus\{x\}$ is enough to construct the 
$\cG_{t,X}$--torsor $E_{\rho}(t)$.

\subsection{Polystable parahoric torsors and unitary representations}\label{se10.3}

As before, fix a maximal compact subgroup $K_G$ of $G$.

\begin{theorem}\label{EHT}
Let $(\cE,\, \boldsymbol\theta)$ be a parahoric 
$\cg_{_{\boldsymbol\theta,X}}$--torsor on $X$ with arbitrary real weights 
$\boldsymbol\theta \,\in\, \ca_{_T}^{^m}$. Then $(\cE,\,\boldsymbol\theta)$ is 
polystable if and only if $(\cE,\, \boldsymbol\theta)$ is given by a homomorphism 
from $\pi_1(X\setminus D)$ to $K_G$ as described in Section \ref{se10.2}.
\end{theorem}

\begin{proof}
First assume that $(\cE ,\,\boldsymbol\theta)$ is given by a homomorphism
$$
\beta\, :\, \pi_1(X\setminus D)\, \longrightarrow\, K_G\, .
$$
Let $\rho\,: \,G \,\lra\, {\rm GL}(V)$ be any homomorphism. Fix a maximal compact
subgroup $K_{{\rm GL}(V)}$ of ${\rm GL}(V)$ such that $\beta(K_G)\, \subset\,
K_{{\rm GL}(V)}$. Then the parabolic vector bundle $W_*$ associated to $(\cE\ ,\,\boldsymbol\theta)$ for
$\rho$ is given by the homomorphism $$\rho\circ\beta\, :\, 
\pi_1(X\setminus D)\, \longrightarrow\, K_{{\rm GL}(V)}\, .
$$
Therefore, this associated parabolic vector bundle $W_*$ is polystable.

To prove the converse, assume that $(\cE ,\,\boldsymbol\theta)$ is polystable. Let 
$\cE({\mathfrak g})$ and $\cE({\mathfrak g}\otimes{\mathfrak g})$ be the parabolic 
vector bundles associated to $(\cE ,\,\boldsymbol\theta)$ for the $G$--modules 
$\mathfrak g$ and ${\mathfrak g}\otimes{\mathfrak g}$ respectively. From Definition 
\ref{polystability} we know that both $\cE({\mathfrak g})$ and $\cE({\mathfrak 
g}\otimes{\mathfrak g})$ are polystable of parabolic degree zero. If 
$\cE({\mathfrak g})$ is given by a homomorphism $\beta$ from $\pi_1(X\setminus D)$ 
to a maximal compact subgroup of $\text{GL}({\mathfrak g})$, then $\cE({\mathfrak 
g}\otimes{\mathfrak g})$ is given by $\beta\otimes\beta$.

For any two parabolic vector bundles given by unitary representations of
$\pi_1(X\setminus D)$, any homomorphism between them is given by a homomorphism of
$\pi_1(X\setminus D)$--modules. Let
$$
\gamma\, :\, \cE({\mathfrak g}\otimes{\mathfrak g})\, \longrightarrow\,
\cE({\mathfrak g})
$$
be the homomorphism of parabolic vector bundles given by the Lie bracket
${\mathfrak g}\otimes{\mathfrak g}\, \longrightarrow\, {\mathfrak g}$. From
the above statement we conclude that $\gamma$ is given by a
homomorphism of $\pi_1(X\setminus D)$--modules. This implies that the connection
on $\cE({\mathfrak g})$ is induced by a connection on $(\cE,\, \theta)$ (see
Definition \ref{Lieconpar}). Therefore, $(\cE,\, \boldsymbol\theta)$ is given by a
homomorphism from $\pi_1(X\setminus D)$ to $K_G$.
\end{proof}

A homomorphism $\rho\, :\, \pi_1(X\setminus D)\, \longrightarrow\, K_G$ is called 
irreducible if $\rho(\pi_1(X\setminus D))$ is not contained in
some proper parabolic subgroup of $G$. A homomorphism $\rho$ is irreducible if and 
only if the space of invariants in ${\mathfrak g}$ for the adjoint action of 
$\rho(\pi_1(X\setminus D))$ is the zero element.

\begin{corollary}\label{EHT2}
Let $(\cE,\, \boldsymbol\theta)$ be a parahoric 
$\cg_{_{\boldsymbol\theta,X}}$--torsor on $X$ with arbitrary real weights
$\boldsymbol\theta \,\in\, \ca_{_T}^{^m}$. Then $(\cE,\,\boldsymbol\theta)$ is 
stable if and only if $(\cE,\, \boldsymbol\theta)$ is given by an irreducible
homomorphism from $\pi_1(X\setminus D)$ to $K_G$ as described in Section \ref{se10.2}.
\end{corollary}

\begin{proof}
Assume that $(\cE,\, \boldsymbol\theta)$ is stable.
Therefore, $(\cE,\, \boldsymbol\theta)$ is polystable.
If the homomorphism $\rho\, :\, \pi_1(X\setminus D)\, \longrightarrow\, K_G$
corresponding to $(\cE,\, \boldsymbol\theta)$
has the property that $\rho(\pi_1(X\setminus D))$ is contained in a proper 
parabolic subgroup $P$ of $G$, then the reduction of $\cE$ to $P$ over $X\setminus 
D$ contradicts the stability of $(\cE,\, \boldsymbol\theta)$. Therefore, $\rho$ is 
irreducible.

Conversely, for a polystable $(\cE,\, \boldsymbol\theta)$, if the
corresponding homomorphism $\rho\, :\, \pi_1(X\setminus D)\, \longrightarrow\, 
K_G$ is irreducible, then the polystable parabolic vector bundle $\cE({\mathfrak g})$
does not admit any holomorphic section \cite[p.~744, Theorem~3]{Si}. Consequently,
the polystable torsor $(\cE,\, \boldsymbol\theta)$ is stable.
\end{proof}

\begin{remark}
The Corlette--Donaldson--Hitchin--Simpson correspondence between flat $G$--bundles and $G$--Higgs
bundles also extends to the parahoric case. When $G \,=\, \text{GL}(n,{\mathbb C})$ this was proved by
Simpson in \cite{Si}. This result of \cite{Si} is the key ingredient in this extension for
general $G$. Using this result the question for $G$ is reduced to one on vector bundles using the
adjoint representation of $G$. The approach in the present paper then goes through without any essential
difficulty.
\end{remark}

\begin{remark}
The paper \cite{biqetc} considers the problem of parabolic Higgs $G$--bundle and 
the Corlette--Donaldson--Hitchin--Simpson correspondence on curves from a somewhat
different perspective and also consider real representations.
\end{remark}

\begin{remark}
The Atiyah--Weil criterion, \cite{At}, \cite{We}, \cite{AB}, for the existence of a
holomorphic connection on a holomorphic principal $G$--bundle generalizes to
${\mathcal G}$--torsors. The proof in \cite{AB} has a straight-forward generalization.
Similarly, the Atiyah--Krull--Schmitt reduction of a holomorphic principal $G$--bundle,
\cite{BBN}, generalizes to ${\mathcal G}$--torsors.
\end{remark}

\begin{remark}
Theorem \ref{EHT} evidently generalizes to the situation where 
$G$ is a product of simple and simply connected groups. The more general case of
semisimple groups $G$ that are not simply connected is covered by using twisted
bundles as in \cite{BLS}. For a reductive group $G$, the natural map
$G\, \longrightarrow\, G/Z_0(G)\times (G/[G,\, G])$ is surjective with finite kernel,
where $Z_0(G)$ is the connected component of the center of $G$ containing the identity
element. Since $G/Z_0(G)$ is semisimple and $/[G,\, G]$ is a product of copies of ${\mathbb G}_m$,
to prove Theorem \ref{EHT} it suffices to prove it for ${\mathbb G}_m$. But this was
done in \cite{Si}.
\end{remark}

\subsection{The reductive case}

We now indicate briefly how to extend the consideration of (semi)stability of torsors in the case 
when the structure group $G$ is a connected reductive algebraic group and identify it with the 
space of homomorphisms from $\pi$ to $K_G$. However, the corresponding relationship with 
parahoric torsors for reductive $G$ needs a closer analysis of Bruhat--Tits theory for reductive 
groups.

Let $S\,= \,[G,\, G]$ be the derived group, i.e. {\em the maximal connected semisimple subgroup} 
of $G$. Let $Z_0$ be the connected component of the center of $G$ (which is a torus) and one know 
that $S$ and $Z_0$ together generate $G$. Let $H\,=\, Z_0 \times S$. Then in fact, $H \, 
\longrightarrow\, G$ is a finite covering map. It is easy to see (following \cite[page 145]{Ra}) 
that $(\Gamma,H)$--bundles gives rise to $(\Gamma,G)$--bundles and the stability and 
semistability of the associated $(\Gamma,G)$--bundles follows immediately from that of the 
$(\Gamma,H)$--bundles.

Thus, the problem of handling the reductive group $H$ reduces to the problem of handling the 
semisimple group $G$ but which is {\em not simply connected}. On the side of Bruhat--Tits group 
schemes and parahoric group schemes, for a general connected reductive $G$ the existence of 
Bruhat--Tits group schemes are well known and would give the existence of similar group schemes on 
the whole of $X$. Several technical issues which one has avoided are in the setting of the 
Bruhat--Tits buildings. Canonical choices of apartments and alcoves which give a transparent 
meaning to the association of conjugacy classes with weights in the alcove would need technical 
modifications which led us too far afield; future considerations would therefore need a careful 
discussion on a ``canonical" choice of apartment as for example indicated in (\cite[page 
32]{tits}).

\section*{Acknowledgements}

We are very grateful to the referee for very detailed comments. A portion of Section 
\ref{se.or} was formulated by the referee. The first two authors acknowledge partial 
support by the J. C. Bose Fellowship. The first author thanks Department of Mathematics, 
Pondicherry University for its hospitality where this work was completed.

\end{document}